\documentclass[a4paper,10pt,reqno]{amsart}

\allowdisplaybreaks[4]

\usepackage{indentfirst} 
\usepackage{amssymb,microtype}
\usepackage{mathtools} 
\usepackage{mathabx} 
\usepackage{amsthm}
\usepackage{thmtools}
\usepackage{bbm}       
\usepackage{enumitem} 
\usepackage[colorlinks=true]{hyperref} 
\usepackage[usenames,dvipsnames]{xcolor} 
\usepackage{ifthen}
\usepackage{tikz}
\usetikzlibrary{decorations.pathreplacing}
\usepackage{tikz-cd}
\usepackage{csquotes}

\makeatletter

\def\MRbibitem{\@ifnextchar[\my@lbibitem\my@bibitem}

\def\mybiblabel#1#2{\@biblabel{{\hyperref{http://www.ams.org/mathscinet-getitem?mr=#1}{}{}{#2}}}}

\def\myhyperanchor#1{\Hy@raisedlink{\hyper@anchorstart{cite.#1}\hyper@anchorend}}

\def\my@lbibitem[#1]#2#3#4\par{%
  \item[\mybiblabel{#2}{#1}\myhyperanchor{#3}\hfill]#4%
  \@ifundefined{ifbackrefparscan}{}{\BR@backref{#3}}%
  \if@filesw{\let\protect\noexpand\immediate
    \write\@auxout{\string\bibcite{#3}{#1}}}\fi\ignorespaces%
}

\def\my@bibitem#1#2#3\par{%
  \refstepcounter\@listctr
  \item[\mybiblabel{#1}{\the\value\@listctr}\myhyperanchor{#2}\hfill]#3%
  \@ifundefined{ifbackrefparscan}{}{\BR@backref{#2}}%
  \if@filesw\immediate\write\@auxout
    {\string\bibcite{#2}{\the\value\@listctr}}\fi\ignorespaces%
}

\makeatother

\DeclareFontFamily{U} {MnSymbolA}{}
\DeclareFontShape{U}{MnSymbolA}{m}{n}{
   <-6> MnSymbolA5
   <6-7> MnSymbolA6
   <7-8> MnSymbolA7
   <8-9> MnSymbolA8
   <9-10> MnSymbolA9
   <10-12> MnSymbolA10
   <12-> MnSymbolA12}{}
\DeclareFontShape{U}{MnSymbolA}{b}{n}{
   <-6> MnSymbolA-Bold5
   <6-7> MnSymbolA-Bold6
   <7-8> MnSymbolA-Bold7
   <8-9> MnSymbolA-Bold8
   <9-10> MnSymbolA-Bold9
   <10-12> MnSymbolA-Bold10
   <12-> MnSymbolA-Bold12}{}
\DeclareSymbolFont{MnSyA} {U} {MnSymbolA}{m}{n}
 \DeclareFontFamily{U} {MnSymbolC}{}
\DeclareFontShape{U}{MnSymbolC}{m}{n}{
  <-6> MnSymbolC5
  <6-7> MnSymbolC6
  <7-8> MnSymbolC7
  <8-9> MnSymbolC8
  <9-10> MnSymbolC9
  <10-12> MnSymbolC10
  <12-> MnSymbolC12}{}
\DeclareFontShape{U}{MnSymbolC}{b}{n}{
  <-6> MnSymbolC-Bold5
  <6-7> MnSymbolC-Bold6
  <7-8> MnSymbolC-Bold7
  <8-9> MnSymbolC-Bold8
  <9-10> MnSymbolC-Bold9
  <10-12> MnSymbolC-Bold10
  <12-> MnSymbolC-Bold12}{}
\DeclareSymbolFont{MnSyC} {U} {MnSymbolC}{m}{n}

\DeclareMathSymbol{\top}{\mathord}{MnSyA}{219} 
\DeclareMathSymbol{\plus}{\mathord}{MnSyC}{20} 

\declaretheorem[numberwithin=section]{theorem}
\declaretheorem[sibling=theorem]{lemma}
\declaretheorem[sibling=theorem]{corollary}

\declaretheorem[sibling=theorem,style=definition]{definition}
\declaretheorem[sibling=theorem,style=remark]{example}

\declaretheorem[sibling=theorem,style=remark]{remark}

\newtheorem{claim}{Claim}

\hypersetup{bookmarksdepth = 3} 
\numberwithin{equation}{section}     

\setlist[enumerate,1]{label={\upshape(\alph*)},ref=\alph*}
\setlist[enumerate,2]{label={\upshape(\arabic*)},ref=\arabic*}

\newcommand{\R}{\mathbb{R}}
\newcommand{\Z}{\mathbb{Z}}
\newcommand{\N}{\mathbb{N}}
\newcommand{\E}{\mathbb{E}}

\def\eps{\varepsilon}
\def\phi{\varphi}
\def\R{{\mathbb R}}

\def\N{{\mathbb N}}
\def\Z{{\mathbb Z}}

\def\le{\leqslant}
\def\ge{\geqslant}

\def\m{\mathbb{M}}
\def\1{ \mathbbm{1}}

\newcommand{\vertiii}[1]{{\left\vert\kern-0.25ex\left\vert\kern-0.25ex\left\vert #1 
    \right\vert\kern-0.25ex\right\vert\kern-0.25ex\right\vert}}
\newcommand{\invertiii}[1]{{\vert\kern-0.25ex\vert\kern-0.25ex\vert #1 
    \vert\kern-0.25ex\vert\kern-0.25ex\vert}}

\begin{document}

\title{Almost sure orbits closeness}
\date{\today}

\author[M. Kirsebom]{Maxim Kirsebom}
\address{Maxim Kirsebom, University of Bremen, Department 3 – Mathematics, Institute for Dynamical Systems, Bibliothekstr. 5, 28359 Bremen, Germany}
\email{kirsebom@uni-bremen.de}

\author[P. Kunde]{Philipp Kunde}
\address{Philipp Kunde, Department of
  Mathematics, Oregon State University, Kidder Hall, Corvallis, OR 97331, USA} 
\email{kundep@oregonstate.edu}
\urladdr{https://sites.google.com/view/pkunde}

\author[T. Persson]{Tomas Persson}
\address{Tomas Persson, Centre for Mathematical Sciences, Lund University, Box 118, 221 00 Lund, Sweden} 
\email{tomasp@gmx.com}
\urladdr{https://tomaspersson.ouvaton.org}

\author[M. Todd]{Mike Todd}
\address{Mike Todd\\ Mathematical Institute\\
University of St Andrews\\
North Haugh\\
St Andrews\\
KY16 9SS\\
Scotland} \email{m.todd@st-andrews.ac.uk}
\urladdr{https://mtoddm.github.io/}

\thanks{We acknowledge financial support by the Hamburg--Lund Funding Program 2023--2025 which made several mutual research visits possible.}



\begin{abstract}
  We consider the minimal distance between orbits of measure
  preserving dynamical systems. In the spirit of dynamical
  shrinking target problems we identify distance rates for which
  almost sure asymptotic closeness properties can be
  ensured. More precisely, we consider the set $E_n$ of pairs of
  points whose orbits up to time $n$ have minimal distance to
  each other less than the threshold $r_n$. We obtain bounds on
  the sequence $(r_n)_n$ to guarantee that $\limsup_{n}E_n$ and
  $\liminf_{n} E_n$ are sets of measure 0 or 1. Results for the
  measure 0 case are obtained in broad generality while the
  measure one case requires assumptions of exponential mixing for
  at least one of the systems. We also consider the analogous
  question of the minimal distance of points within a single
  orbit of one dimensional exponentially mixing dynamical systems. 
\end{abstract}

\maketitle

\section{Introduction}

In a metric space $(X,d)$, the problem of the shortest distance
between two orbits of a dynamical system $T \colon X\to X$, with an
ergodic measure $\mu$, was introduced in \cite{BarLiaRou19}. That
is, for $n\in \N$ and $x, y\in X$, they studied
\begin{equation}\label{eq:min1}
  \m_n(x, y)= \m_{T, n}(x, y):=\min_{0\le i,j<n} d(T^i(x),
  T^j(y))
\end{equation}
and showed that the decay of $\m_n$ depends on the correlation
dimension.

The \emph{lower correlation dimension} of $\mu$ is defined by 
\[
  \underline{C}_\mu:=\liminf_{r\to 0} \frac{\log\int\mu(B(x,
    r)) \, d\mu(x)}{\log r},
\]
and the \emph{upper correlation dimension} $\overline{C}_\mu$ is
analogously defined via the limsup.  If these are equal, then
this is $C_\mu$, the \emph{correlation dimension} of $\mu$. This
dimension plays an important role in the description of the
fractal structure of invariant sets in dynamical systems and has
been widely studied from different points of view: numerical
estimates (e.g.\ \cite{BadBro88, BesPalTurVai88, SprRow01}),
existence and relations with other fractal dimension (e.g.\
\cite{BarGerTch01, Pes93}) and relations with other dynamical
quantities (e.g.\ \cite{FarVai18, Man10}).

In \cite[Theorem~1]{BarLiaRou19}, under the assumption
$\underline{C}_\mu>0$, a general upper bound for $\m_n$ was
obtained:
\begin{theorem}For a dynamical system $(X,T,\mu)$, we have
\begin{equation*}
  \limsup_{n}\frac{\log \m_{T, n}(x, y)}{-\log n} \le
  \frac{2}{\underline{C}_\mu} \qquad (\mu\times\mu)\text{-a.e. } x,
  y. 
\end{equation*}
\label{thm:BLR}
\end{theorem}

To replace the inequality above with equality, in
\cite[Theorems~3 and 6]{BarLiaRou19}, the authors assumed that
$C_\mu$ exists and proved
\begin{equation}
  \liminf_{n}\frac{\log \m_{T,n}(x, y)}{-\log n} \ge
  \frac{2}{C_\mu} \qquad (\mu\times\mu)\text{-a.e. } x,
  y,\label{eq:WTS}
\end{equation}
using some exponential mixing conditions on the system.  This was
shown in \cite{RouTod24} to be unnecessary in cases where there
is an appropriate inducing scheme.

In the proofs of the above theorems, the approach was to find a
sequences $(r_n)_n$ and show, using the Borel--Cantelli Lemma,
that for almost every $(x, y)$ for all large enough $n$ either
$\m_{T, n}(x, y)\le r_n$, or $\m_{T, n}(x, y)\ge r_n$.  The
sequences were, for any $\eps>0$, of the form
$r_n = \frac1{(n^2\log n)^{(\underline{C}_\mu-\eps)}}$ in the
former case and
$r_n = \frac1{(n^2(\log n)^b)^{(\overline{C}_\mu+\eps)}}$, for
some $b<-4$, in the latter.  Our aim in the current work is to
refine these estimates and extend their applications.

\textit{Notation:}  we will sometimes use, for non-negative real sequences $(a_n)_n, (b_n)_n$,  the notation $a_n\lesssim b_n$ to mean that there is some $C>0$ such that $a_n\le C b_n$ for all $n$, similarly for $\gtrsim$.  If  $a_n\gtrsim b_n$ and  $a_n\lesssim b_n$, we write $a_n\asymp b_n$.

\section{Main results}
We begin with the most general setup which consists of two
probability preserving dynamical systems $(T_1,\mu_1)$ and
$(T_2,\mu_2)$ on the same space $X$, i.e.\
$T_1, T_2 \colon X\to X$. We let $(r_n)_n$ denote a sequence of
positive numbers. Note that in this setting
$\int \mu_1 (B(y, r_{n})) \, d\mu_2(y)= \int \mu_2 (B(y,
r_{n}))\, d\mu_1(y)$. We define the minimum analogously to
\eqref{eq:min1},
\[
  \m_n(x, y)= \m_{(T_1,T_2), n}(x, y):=\min_{0\le i,j< n}
  d(T_1^i(x), T_2^j(y)),
\]
and define 
\begin{align*}
  E_n=E_{n,r_n}^{T_1, T_2}:
  &= \bigl\{\, (x,y)\in X\times X: \m_{(T_1,T_2), n}(x, y)<
    r_n\}\\
  &= \bigl\{\, (x, y)\in X\times X: d(T_1^ix, T_2^jy)<r_n  \text{
    for some } 0 \leq i,j <n \, \bigr\}.
\end{align*}
Henceforth we abandon the use of $\m_n$ and instead use the
latter form of the above equation.  Ultimately one would like
conditions on the choices for $(r_n)_n$ of the following form:
\[
  \text{ Condition on } (r_n)_n \Longleftrightarrow
  (\mu_1\times\mu_2)(\liminf_n E_{n,r_n}^{T_1, T_2}) = 1,
\]
i.e., for almost every $(x, y)$, for all large enough $n$ there
exist $0\le i,j<n$ such that $d(T_1^ix, T_2^jy)<r_n$; and
\[
  \text{ Condition on } (r_n)_n \Longleftrightarrow
  (\mu_1\times\mu_2)(\limsup_n E_{n,r_n}^{T_1, T_2}) = 1,
\]
i.e., almost surely there are infinitely many $n$ such that there
exist $0\le i,j<n$ with $d(T_1^ix, T_2^jy)<r_n$.

Some results will be stated explicitely for the special case when
$T_1=T_2$ and $\mu_1=\mu_2$ in which case we simply talk about
the probability preserving dynamical system $(X,T,\mu)$ and write
\[
  E_n=E_{n,r_n}^T:= \bigl\{ \, (x, y)\in X\times X: d(T^ix,
  T^jy)<r_n \text{ for some } 0 \leq i,j < n \,\bigr\}.
\]
In the special case of one dynamical system we will also consider
the further specialized case of $x=y$, or in other words, the
case of one single orbit's minimal internal distance. In this
case we change notation and write
\[
  F_n=F_{n,r_n}^T := \{\, x\in X : d(T^i x, T^j x) < r_n \text{
    for some } 0 \leq i< j < n \,\}.
\]
The $\limsup$ and $\liminf$ results that we wish to obtain remain
analogous for this set.

The main results, which are to follow, are initially split into
the case of two distinct orbits and the case of one single orbit.

For two distinct orbits of two dynamical systems we obtain
conditions on measure zero for both $\liminf$ and $\limsup$ sets
in a very general setting. Specialising to two mixing systems we
obtain conditions for the $\liminf$ and $\limsup$ sets to have
measure one and simplify these conditions in a corollary when the
two mixing systems are indeed identical. In the special case of
one dynamical system which is chosen to be the doubling map we
get particularly sharp results. For the $\limsup$ set we obtain a
dichotomy and for the $\liminf$ set we obtain bounds sharp enough
to produce a shrinking rate of the $(r_n)_n$ which gives measure
zero for the $\liminf$ set but measure one for the $\limsup$
set. Our final result within the two distinct orbit case concerns
the situation when one system is a rotation with a Diophantine
condition and one system is mixing with respect to the Lebesgue
measure. In this case we also obtain conditions for which the
$\liminf$ and $\limsup$ sets have measure one.

The single orbit case is more difficult and we obtain conditions
for measure zero and one of the $\liminf$ and $\limsup$ sets
under somewhat stronger mixing assumptions.

\subsection{Results for two distinct orbits}

\begin{theorem}\label{the:measurezero}
  Let $(X,T_1,\mu_1)$ and $(X,T_2,\mu_2)$ be two probability
  preserving dynamical systems and $(r_n)_n$ a sequence of
  positive numbers. Then
  \begin{enumerate}
  \item \label{infzero} If $n^2 \int \mu_1 (B(y,r_n))\,d\mu_2(y)\to 0$, then
    \[(\mu_1 \times \mu_2) (\liminf_{n} E_{n,r_n}^{T_1,T_2}) =
      0.\]
  \item \label{supzero} If $n\int \mu_1 (B(y, r_{n}))d\mu_2(y)$
    is decreasing and
    $\sum_{n=1}^\infty n\int \mu_1 (B(y, r_{n})) \,
    d\mu_2(y)<\infty$ then
    \[(\mu_1\times\mu_2) (\limsup_n E_{n,r_n}^{T_1, T_2}) = 0.\]
  \end{enumerate}
\end{theorem}
We note that the result holds with the obvious simplifications
when the two dynamical systems are identical.  For part
(\ref{supzero}), the summability condition on the $(r_n)_n$ holds
for example if
$\int\mu_1(B(y, r_n)) \, d\mu_2(y)\lesssim 1/(n^2 (\log
n)^{1+\eps})$.  Note that in the case of one dynamical system,
criteria of this type appear in other results, such as
\cite[Theorem~C]{KirKunPer23}, where
$\sum_{n=1}^\infty \int\mu(B(y, r_n)) \, d\mu(y) <\infty$ is
assumed in order to get a result on returns of $x$ to itself.

We will need the following definition of exponential mixing wrt.\
observables in the well-known function spaces $BV$ and
$L^{\infty}$.
\begin{definition}
  Let $(X, T, \mu)$ denote a measure preserving system where $X\subset \R$. If there are
  $C, \theta>0$ such that for all $\psi\in BV$ and
  $\phi\in L^\infty$,
  \[
    \biggl|\int \psi\cdot \phi\circ T^n \, d\mu- \int \psi \,
    d\mu \int \phi \, d\mu \biggr| \le
    C\|\psi\|_{BV}\|\phi\|_{L^\infty} e^{-\theta n},
  \]
  then we say that $(X, T, \mu)$ has \emph{exponential mixing for
    $BV$ against $L^\infty$}.
\end{definition}
In Section~\ref{ssec:single}, some results will require $L^\infty$ to be replaced by $L^1$ in this definition. Systems that are known to satisfy exponential mixing for $BV$
against either $L^1$ or $L^{\infty}$ include piecewise expanding interval maps
with $\mu$ being a Gibbs measure and quadratic maps with
Benedicks--Carleson parameter and $\mu$ being the absolutely
continuous invariant measure (\cite{LivSauVai98},
\cite{You92}).

The next definition concerns exponential mixing wrt.\ observables
in $L^{\infty}$ and the less-known function space
$V_{\alpha}$. This space along with its associated norm
$|\cdot|_{\alpha}$ was introduced by Saussol in \cite{Sau00}
where he defined and studied multidimensional piecewise expanding
maps (see also \cite[section~5]{BarLiaRou19}). The definition of these maps
and of $V_{\alpha}$ is rather involved and extensive, hence we
refrain from giving the details and refer the interested reader
to the cited papers. For our purposes the intuition that the
$|\cdot|_{\alpha}$-norm is an analogue of the $BV$-norm for
higher dimensional maps suffices (in fact, here we only use that characteristic function on balls have norm bounded independently of the ball, as well as the maps $x\mapsto \mu(B(x, r))$) along with the fact that Saussol
proved that his multidimensional piecewise expanding maps defined
on a compact subset of $\R^n$ satisfy the following mixing
property.
\begin{definition}[see \cite{Sau00} and \cite{BarLiaRou19}]\label{def:mixingalphaLinfty}
  Let $(X, T, \mu)$ denote a measure preserving system where
  $X\subset \R^n$. If there are $C, \theta>0$ such that for all
  $\psi\in V_{\alpha}$ and $\phi\in L^\infty$,
  \[
    \biggl|\int \psi\cdot \phi\circ T^n \, d\mu- \int \psi \,
    d\mu \int \phi \, d\mu \biggr| \le
    C|\psi|_{\alpha}\|\phi\|_{L^\infty} e^{-\theta n},
  \]
  then we say that $(X, T, \mu)$ has \emph{exponential mixing for
    $V_{\alpha}$ against $L^\infty$}.
\end{definition}

\begin{theorem}
  \label{thm:supinf2} Let $(X,T_1,\mu_1)$ and $(X,T_2,\mu_2)$ be
  two probability preserving systems which are exponentially
  mixing for either $BV$ against $L^\infty$ or $V_{\alpha}$
  against $L^{\infty}$. Let $(r_n)_n$ be a sequence of positive
  numbers.
  \begin{enumerate}  
  \item \label{thm:supinf2-1} If $(r_n)_n$ is decreasing and for
    some $\eps>0$ we have
    \[
      \int\mu_1(B(y, r_n)) \, d\mu_2(y) \gtrsim \frac{(\log
        n)^3(\log\log n)^{1+\eps}}{n^2}
    \]
    and
    \begin{equation}
      \label{eq:extraconditionA}
      \frac{\bigl(\int\mu_i (B(y, r_n)) \, d\mu_i (y)
        \bigr)^{\frac12}}{\int\mu_1(B(y, r_n)) \,
        d\mu_2(y)} \lesssim \frac{n}{(\log n)^2(\log\log
        n)^{1+\eps}}, \quad i = 1,2,
    \end{equation}
    then
    \[(\mu_1\times\mu_2) \left(\liminf_n E_{n,r_n}^{T_1, T_2}\right) = 1.\]
    
  \item \label{thm:supinf2-2} If for some $h(n)\to\infty$ we have
    \[
      \int\mu_1(B(y, r_n)) \, d\mu_2(y) \geq \frac{(\log n)^2
        h(n)}{n^2}
    \]
    and
    \begin{equation}
      \label{eq:extraconditionB}
      \frac{\bigl(\int\mu_i (B(y, r_n)) \, d\mu_i (y)
        \bigr)^{\frac12}}{\int\mu_1(B(y, r_n)) \, d\mu_2(y)}
      \leq \frac{n}{(\log n)h(n)}, \quad i = 1,2,
    \end{equation}
    then
    \[(\mu_1\times\mu_2) \left(\limsup_n E_{n,r_n}^{T_1, T_2}\right) = 1.\]
  \end{enumerate}
\end{theorem}

\begin{remark}
  If $\mu_1$ or $\mu_2$ are Ahlfors regular, i.e., there exist
  $C, s>0$ such that
  \[
    \frac1Cr^s\le \mu(B(x, r))\le Cr^s \text{ for all } x,
  \]
  then all of our conditions on measures of balls, and their
  integrals, depend only on $s$ and $(r_n)_n$.  
\end{remark}

We examine the conditions on $(r_n)_n$ in part~(\ref{thm:supinf2-1}) of Theorem~\ref{thm:supinf2} in two examples. Part~(\ref{thm:supinf2-2}) can be considered
in the same way.
\begin{example}
If $\mu_1$ and $\mu_2$ are equivalent and there is
a constant $c \geq 1$ such that
\[
  c^{-1} \leq \frac{d \mu_1}{d \mu_2} \leq c,
\]
then for all large $n$ it is in this case enough, since
$\varepsilon > 0$ is arbitrary, to
require in part~(\ref{thm:supinf2-1}) that
\[
  \int \mu_1 (B(y,r_n)) \, d \mu_1 (y) \geq \frac{(\log
    n)^4 (\log \log n)^{2+\varepsilon}}{n^2},
\]
and
\[
  \int \mu_2 (B(y,r_n)) \, d \mu_2 (y) \geq \frac{(\log
    n)^4 (\log \log n)^{2+\varepsilon}}{n^2}.
\]
\end{example}

\begin{example}
Suppose that $(T_1,\mu_1)$ is
  the doubling map with Lebesgue measure, and that $(T_2,\mu_2)$
  is a quadratic map for a Benedicks--Carleson parameter, with
  the invariant measure which is absolutely continuous with
  respect to Lebesgue measure.

  We have of course that $r \leq \mu_1 (B(y,r)) \leq 2r$, so that
  \[
    r \leq \int \mu_1 (B(y,r)) \, d \mu_1 (y) \leq 2r
  \]
  and
  \[
    r \leq \int \mu_1 (B(y,r)) \, d \mu_2 (y) \leq 2r.
  \]

  To estimate $\int \mu_2 (B(y,r)) \, d \mu_2 (y)$, we note that
  it is known that $\mu_2$ has a density which is sum of a BV function,
  bounded away from zero, and countably many one-sided
  singularities of the form $1/\sqrt{x}$, see \cite[Theorem 1]{You92}. Because of this, we
  have $c r \leq \mu_2 (B(y,r))$ and
  \[
    c r \leq \int \mu_2 (B(y,r)) \, \mathrm{d} \mu_2 (y).
  \]

  To get an upper bound, it is enough to consider the case that
  the density of $\mu_2$ is $h (x) = 1/\sqrt{x}$ on $[0,1]$. We
  then have that
  \[
    \mu_2 (B(y,r)) = \int_{y-r}^{y+r} \frac{1}{\sqrt{x}} \, dx \leq
    \frac{2r}{\sqrt{y-r}}
  \]
  if $y-r \geq 0$ and
  \[
    \mu_2 (B(y,r)) = \int_{0}^{y+r} \frac{1}{\sqrt{x}} \, dx = 2
    \sqrt{y+r} \leq 2 \sqrt{2} \sqrt{r}
  \]
  if $y -r < 0$.

  Hence,
  \begin{align*}
    \int \mu_2 (B(y,r)) \, d \mu_2 (y)
    & \leq \int_0^r 2\sqrt{2r}
      \frac{1}{\sqrt{y}} \, dy + \int_r^1     \frac{2r}{\sqrt{y-r}}
      \frac{1}{\sqrt{y}} \, dy \\
    & \leq 4 \sqrt{2 r} \sqrt{r} + 2r \int_0^1 \frac{1}{\sqrt{y}
      \sqrt{y+r}} \, dy \\
    & = 4 \sqrt{2} r + 4r \log (1/\sqrt{r} + \sqrt{1 +
      1/r}) \\
    & \leq C r |\log r|.
  \end{align*}
  The conditions in Theorem~\ref{thm:supinf2}~(\ref{thm:supinf2-1}) on $r_n$ are therefore the
  following:
  \begin{align*}
    r_n &\geq \frac{(\log n)^3 (\log \log n)^{1+
      \varepsilon}}{n^2},\\
    \frac{\sqrt{r_n}}{r_n} &\leq \frac{n}{(\log n)^2 (\log \log
      n)^{1+\varepsilon}}, \\
    \frac{\sqrt{r_n |\log r_n|}}{r_n} &\leq
      \frac{n}{(\log n)^2 (\log \log n)^{1+\varepsilon}}.
  \end{align*}
  This simplifies to
  \[
    r_n \geq \frac{(\log n)^3 (\log \log n)^{1+
      \varepsilon}}{n^2}, \quad \text{and} \quad \frac{r_n}{|\log
    r_n|} \geq \frac{(\log n)^4 (\log \log n)^{2 + 2
      \varepsilon}}{n^2}.
  \]
  Hence, it is enough to require that
  \[
    r_n \geq \frac{(\log n)^5 (\log \log n)^{2 + 
        \varepsilon}}{n^2},
  \]
  for some $\varepsilon > 0$.
\end{example}

The proof of Theorem~\ref{thm:supinf2} will be given for the case
of exponential mixing for $BV$ against $L^{\infty}$. The proof
for $V_{\alpha}$ against $L^{\infty}$ is essentially identical
but would require us to give the precise definition of
$V_{\alpha}$ and $|\cdot|_{\alpha}$.

As an immediate corollary of Theorem~\ref{thm:supinf2} we obtain the simpler statement for one system.
\begin{corollary}\label{thm:supinf}
  Let $(X,T,\mu)$ be a probability preserving dynamical system
  which is exponentially mixing for either $BV$ against
  $L^\infty$ or $V_{\alpha}$ against $L^{\infty}$. Let $(r_n)_n$
  be a sequence of positive numbers.
  \begin{enumerate}
  \item If $(r_n)_n$ is decreasing and
    \[
      \int\mu(B(y, r_n)) \, d\mu(y) \gtrsim \frac{(\log n)^4(\log\log
        n)^{2+\eps}}{n^2},
    \]
    then 
    \[(\mu\times\mu) (\liminf_n E_{n,r_n}^T) = 1.\]
    \item If for some function $h(n)\to\infty$ 
    \[
    \int\mu(B(y, r_n)) \, d\mu(y) \ge \frac{(\log n)^2 h(n)}{n^2},
    \]
    then 
    \[(\mu\times\mu) (\limsup_n E_{n,r_n}^T) = 1.\]
  \end{enumerate}
\end{corollary}

\subsection*{The doubling map case}

Let $X=[0,1]$, $T \colon X\to X$ be given by $Tx=2x \mod 1$, let
$\mu$ denote the Lebesgue measure on $X$ and let $d$ denote the
Euclidian metric on $X$. We note that all results in this
subsection also hold for the map $Tx=kx \mod 1$, $k\in\N$,
$k \geq 2$.

For $\liminf_{n}E_n$ we get the following result.
\begin{theorem}\label{the:2xliminf} \text{ }
  Let $(r_n)_n$ be a sequence of positive numbers.
  \begin{enumerate}
  \item \label{2xliminf0} If $n^2 r_n\to 0$ then 
    \[(\mu \times \mu)\left(\liminf_{n} E_{n,r_n}^T\right)=0.\]
  \item \label{2xliminf1}\label{prop:liminf1} If $(r_n)_n$ is decreasing and 
    \begin{equation}\label{eq:liminf1}
      \sum_{n=0}^{\infty} \frac{1}{2^{2n}r_{2^n}}<\infty.
    \end{equation}
    then
    \begin{equation*}
      (\mu \times \mu)\left(\liminf_{n} E_{n,r_n}^T\right)=1.
    \end{equation*}
  \end{enumerate}
\end{theorem}
\begin{remark}
  We note that by Cauchy condensation (see
  Lemma~\ref{lem:CauchyC} below), we have
  \[
    \sum_{n = 1}^\infty \frac{1}{n^3 r_n} < \infty \quad
    \Leftrightarrow \quad \sum_{n=0}^{\infty}
    \frac{1}{2^{2n}r_{2^n}}<\infty,
  \]
  provided $(n^3 r_n)_n$ is increasing. Condition
  \eqref{eq:liminf1} is for example satisfied for
  \begin{equation*}
    r_n= \frac{\log n(\log\log n)^{1+\epsilon}}{n^2}
  \end{equation*} 
  for any $\epsilon>0$. In this case we see that
  \[
    \sum_{n=0}^{\infty} \frac{1}{n^3 r_{n}} =\sum_{n=0}^{\infty}
    \frac{1}{n \log n (\log \log n)^{1+\epsilon}} < \infty.
  \]
\end{remark}

We are able to prove the following dichotomy for
$\limsup_{n} E_n$.
\begin{theorem}\label{thm:2xmod1Dicho}
  Let $(r_n)_n$ be a decreasing sequence of positive numbers
  s.t.\ also $(nr_n)_n$ is decreasing. Then
  \begin{equation*}
    (\mu \times \mu)(\limsup_{n}E_{n,r_n}^T)=\begin{cases}
      0 & \text{if } \sum_{n=1}^{\infty} nr_n<\infty,\\
      1 & \text{if } \sum_{n=1}^{\infty} nr_n=\infty.
    \end{cases}
  \end{equation*}
\end{theorem}

\begin{remark}
  Actually, $(nr_n)_n$ being decreasing is only required for the
  measure 0 case. For the measure 1 case it is sufficient to
  assume that $(r_n)_n$ is decreasing.
\end{remark}

\begin{remark}
  We note that Theorem~\ref{the:2xliminf} and
  Theorem~\ref{thm:2xmod1Dicho} together show that radii
  sequences with
  $0=(\mu \times \mu)(\liminf_{n} E_n)<(\mu \times \mu)(\limsup_{n}
  E_n)=1$ exist. Take for example
  \begin{equation*}
    r_n=\frac{1}{n^2 \log n}.
  \end{equation*}
  Then $n^2r_n=\frac{1}{\log n}\to 0$ and
  $\sum_{n=1}^{\infty}nr_n=\sum_{n=1}^{\infty}\frac{1}{n\log
    n}=\infty$ which means that the conditions for both results
  are satisfied.
\end{remark}

\subsection{Results when one system is a rotation}

\begin{theorem}
   Suppose that $X$ is the circle and both $\mu_1$ and $\mu_2$ are
  Lebesgue measure. Moreover, suppose that $(X,T_1,\mu_1)$ is
  exponentially mixing for $BV$ against $L^{\infty}$ and that
  $(X,T_2,\mu_2)$ is a rotation by an angle $\alpha$. Let
  $(r_n)_n$ be a sequence of positive numbers.
  \begin{enumerate}
  \item \label{thm:rotinfzero} If $n^2 r_n\to 0$, then
    $(\mu_1 \times \mu_2) (\liminf_{n\to\infty}
    E_{n,r_n}^{T_1,T_2}) = 0.$
  \item \label{thm:rotsupzero} If $(n r_n)_n$ is decreasing and
    $\sum_{n=1}^\infty nr_n<\infty$, then
    $(\mu_1\times\mu_2) (\limsup_n E_{n,r_n}^{T_1,T_2}) = 0.$
  \item \label{thm:rotinfone} Suppose some $\eps>0$ that $\alpha$
    satisfies
    $$|q \alpha - p| \geq c(\alpha) (\log q)^2 \cdot
    (\log\log(q))^{1+\eps}/ q^{2}$$ for all sufficiently large
    $q \in \mathbb{N}$. If $(r_n)_n$ is decreasing and satisfies
    $r_n \gtrsim \frac{(\log n)^2(\log\log n)^{1+\delta}}{n^2}$ with
    $0<\delta<\eps$, then
    $(\mu_1\times\mu_2) (\liminf_n E_{n,r_n}^{T_1,T_2}) = 1$.
  \item \label{thm:rotsupone} Suppose $\alpha$ satisfies
    $|q \alpha - p| \geq c(\alpha) \log q \cdot \phi(q) / q^{2}$
    for all sufficiently large $q \in \mathbb{N}$, where
    $\phi(q) \to \infty$ as $q \to \infty$. If $(r_n)_n$
    satisfies $r_n \geq \frac{\log n \cdot h(n)}{n^2}$ for some
    $h(n)$ with $h(n) \to \infty$ and $h(n)/\phi(n) \to 0$ as
    $n \to \infty$, then
    $(\mu_1\times\mu_2) (\limsup_n E_{n,r_n}^{T_1,T_2}) = 1$.
  \end{enumerate}
  \label{thm:rotation}
\end{theorem} 
We emphazise that the Diophantine condition imposed on the angle
$\alpha$ in parts (\ref{thm:rotinfone}) and (\ref{thm:rotsupone})
is satisfied for almost every $\alpha$. This follows from
Khinchin's Theorem on metric Diophantine approximation which
implies that if $\sum_q \psi(q)$ converges, then a.e.\ number
$\alpha$ satisfies $|q\alpha - p|\geq \psi(q)$ for all
sufficiently large $q\in\N$. For example, our Diophantine
conditions are satisfied for all numbers of Diophantine exponent
strictly less than 1. Here we recall that $\alpha$ has
Diophantine exponent $\sigma$ if there exists a constant $C>0$
such that $|q\alpha-p|\geq \frac{C}{q^{1+\sigma}}$ for all
$q\in\N$.

\subsection{Results for a single orbit}
\label{ssec:single}

Recall the notation
\[
  F_n = F_{n,r_n}^T = \{\, x : d(T^i x, T^j x) < r_n \text{ for
    some } 0 \leq i< j < n \,\}.
\]
We will need the following definition.
\begin{definition}
  Let $(X, T, \mu)$ denote a measure preserving system where $X\subset \R$. If there exist $C', \theta'>0$ such that for all
  $\psi_1, \psi_2\in BV$, $\phi_1, \phi_2\in L^\infty$, and $0
  \leq a < b \leq c$,
  \begin{align*}
    & \biggl | \int \psi_1\cdot \psi_2\circ T^a\cdot \phi_1\circ
      T^b\cdot \phi_2\circ T^c \, d\mu - \int \psi_1\cdot
      \psi_2\circ T^a \, d\mu \int \phi_1\cdot
      \phi_2\circ T^{c-b} \, d\mu \biggr|\\
    &\hspace{8cm} \le C'(\psi_1,\psi_2,\phi_1,\phi_2) e^{-\theta'
      (b-a)},
  \end{align*}
  then we say that $(X, T, \mu)$ has \emph{exponential 4-mixing
    for $BV$ against $L^\infty$}.
\end{definition}
Note that by setting $\psi_1=\phi_2=\mathrm{Id}$ we get 2-mixing for $BV$ against $L^{\infty}$ with the same constants $C'$ and $\theta'$ as in the 4-mixing estimate. 
The 4-mixing property is known to hold for Gibbs--Markov interval maps, see \cite[Lemma~4.16]{Zha24}. 
In our case each of the observables will be characteristic
functions on an interval, and we will assume that $C'$ can be
taken independently of the intervals.  See \cite[Lemma
4.16]{Zha24} for a case of this.

\begin{remark}
The uniformity of $C'$ is automatic for intervals. By
Collet \cite{Col99} (using Banach--Steinhaus), we have
$C' (\psi_1,\psi_2,\phi_1,\phi_2) = C \lVert \psi_1 \rVert_{BV}
\lVert \psi_2 \rVert_{BV} \lVert \phi_1 \rVert_\infty \lVert
\phi_2 \rVert_\infty$, and for intervals all these norms are
uniformly bounded.
\end{remark}

Let $A (r,n) := \{\, x : d (x,T^n x) < r \,\}$. We will need to assume the
following short return time estimate which states that there exist $C$ and $s>0$ such that
\begin{equation}
  \label{eq:shortreturnestimate}
  \mu (A (r,n)) \leq C r^s.
\end{equation}
There are several known estimates of this type in the literature,
see for instance Holland, Nicol and Török \cite[Lemma~3.4]{HNT},
Kirsebom, Kunde and Persson \cite[Section~6.2]{KirKunPer23} and Holland,
Kirsebom, Kunde and Persson \cite[Lemma~13.7]{HolKirKunPer24}.

\begin{theorem}
  Let $(X,T,\mu)$ be probability preserving dynamical system with
  an interval $X\subset \R$. Let $(r_n)_n$ be a sequence of
  positive numbers.
	\begin{enumerate}
		\item\label{thm:1orbinfzero} Assume $(X, T, \mu)$ is exponentially mixing for $BV$
		against $L^1$, and that the short return time estimate
		\eqref{eq:shortreturnestimate} holds with constants
		$C,s > 0$.
		If $r_n \leq 1/((n \log n)^\frac{1}{s} h(n))$ and
		$\int \mu (B(x,r_n)) \, d \mu(x) \leq \frac{1}{n^2 h(n)}$ for
		some function $h(n) \to \infty$ when $n \to \infty$, then
		$\mu (\liminf_n F_{n,r_n}^T) = 0$.

		\item\label{thm:1orbsupzero} Assume $(X,T,\mu)$ is exponentially mixing for $BV$
		against $L^1$, that the short return time estimate
		\eqref{eq:shortreturnestimate} holds with constants $C,s > 0$
		and that $r_n \lesssim n^{-1/s} (\log n)^{-2/s - \varepsilon}$
		for some $\varepsilon$.
		If $n\int \mu(B(x,r_n))\, d\mu(x)$ is decreasing and 
 		$\sum_{n=1}^\infty n \int \mu (B(x,r_n)) \, \mathrm{d} \mu(x) <
		\infty$, then $\mu (\limsup_n F_{n,r_n}^T) = 0$.
		
		\item\label{thm:1orbinfone} Assume $(X, T, \mu)$ is exponentially 4-mixing for $BV$
		against $L^\infty$. Let $(r_n)_n$ be decreasing such that $r_n\ge n^{-\beta}$ for some $\beta>0$
		and
		$\int\mu(B(y, r_n)) \, d\mu(y) \ge \frac{(\log n)^4(\log\log
		n)^{2+\eps}}{n^2}$, then $\mu(\liminf_n F_{n,4r_n}^T) = 1$.
		
		\item\label{thm:1orbsupone} Assume $(X, T, \mu)$ is exponentially 4-mixing for $BV$
		against $L^\infty$, $r_n\ge n^{-\beta}$ for some $\beta>0$
		and
		$\int\mu(B(y, r_n)) \, d\mu(y) \ge \frac{(\log n)^2
			h(n)}{n^2}$, for some function $h(n) \to \infty$ as
		$n \to \infty$, then $\mu(\limsup_n F_{n,4r_n}^T) = 1$.
	\end{enumerate}
	\label{thm:1supinf}
\end{theorem}
\begin{remark}
	Theorem~\ref{thm:1supinf} is stated for $X\subset \R$ being an interval. We could have stated it more generally for a compact subset $X\subset \R^n$ by changing the definition of 4-mixing analogue to Definition~\ref{def:mixingalphaLinfty}, i.e.\ by exchanging the $BV$-norm with the $|\cdot|_{\alpha}$-norm. However, beyond interval maps there are, to our knowledge, no known examples of systems satisfying either this generalized 4-mixing for $V_{\alpha}$ or the short return time estimate of \eqref{eq:shortreturnestimate}. For this reason we opted for the simpler formulation. 
\end{remark}
Examples of systems for which Theorem~\ref{thm:1supinf} holds
include the doubling map (and more generally $Tx=kx \mod 1$ for
some $k\in\N$, $k \geq 2$), piecewise expanding interval maps where $\mu$ is absolutely continuous wrt.\ Lebesgue, and the Gauss map with the Gauss measure.

The proof of Theorem~\ref{thm:1supinf}(\ref{thm:1orbinfone}) is similar to the proof of
\cite[(4.2) in Theorem~4.8]{Zha24}: note that many of the ideas there were
developed for a different case in \cite{GouRouSta24}.

\section{Two mixing systems proofs}

We include a useful version of Cauchy Condensation, see for
example the proof of \cite[Proposition~10.1]{HolKirKunPer24}:

\begin{lemma}
  Suppose $(c_n)_n$ has $c_n\geq c_{n+1}>0$ and $a>1$.  Then
  \[
    \sum_{n=1}^\infty c_n<\infty \quad \Longleftrightarrow \quad
    \sum_{k=1}^\infty a^kc_{\lfloor a^k\rfloor}<\infty.
  \]
  \label{lem:CauchyC}
\end{lemma}

We also require the following results.

\begin{lemma}
  Suppose that $(X,\mu_1)$, $(X,\mu_2)$ are probability spaces with $X\subset \R^n$. Then there exists $K>0$ such that
  if $r>0$ is sufficiently small then
  \[
    \mu_i(B(y, r))\le K \biggl(\int\mu_i(B(x,
      r)) \, d\mu_i(x) \biggr)^{\frac12}\quad \text{for } i=1,2.
  \]
  Hence also
  \[
    \int\mu_i(B(x, r))^2 \, d\mu_j(x) \le K \biggl(\int\mu_i(B(x,
      r)) \, d\mu_i(x)\biggr)^{\frac12}\int\mu_i(B(x,
    r)) \, d\mu_j(x).
  \]
  \label{lem:tight}
  for $(i,j)=(1,2), (2,1)$.
\end{lemma}

Note that when $\mu_1=\mu_2$, the second statement is the
conclusion of \cite[Lemma~14]{BarLiaRou19}.

\begin{proof}
  Since $X\subset \R^n$, we can cover $B(y, r)$ by $K_0$ balls
  of radius $r/2$.  Let $B(z, r/2)$ be the ball of largest $\mu_i$-measure, then
  $\mu_i (B(z, r/2)) \ge \mu_i (B(y, r))/K_0$.  Then
  \begin{align*}
    \int\mu_i (B(x, r)) \, d\mu_i(x)
    & \ge \int_{B(z, r/2)}\mu_i (B(x, r)) \, d\mu_i(x) \ge
      \int_{B(z, r/2)}\mu_i (B(z, r/2)) \, d\mu_i(x)\\ 
    & = \mu_i (B(z, r/2))^2 \ge \frac1{K_0^2} \mu_i (B(y,
      r))^2, 
  \end{align*}
  so setting $K=K_0$ the first statement of the lemma is proved.
  The second statement is then immediate.
\end{proof}

\begin{lemma}
  Let $(X,\mu)$ denote a probability space where $X\subset
  \R$. For any $r>0$, $\psi_r \colon X \to \R$ given by
  $y\mapsto \mu_1(B(y,r))$ is BV with total variation bounded
  above by 2.
  \label{lem:BV}
\end{lemma}

\begin{proof}
  We have that
  $\mu (B(y,r)) = \mu ((-\infty,y + r)) - \mu ((-\infty,
  y-r])$. Hence, the function $\psi_r$ is a difference of two
  increasing functions, both increasing from $0$ to $1$. It
  follows immediately that the total variation of $\psi_r$ is at
  most $1 + 1 = 2$.
\end{proof}

\begin{remark}
  An analogue of this lemma in higher dimension for
  $|\cdot|_{\alpha}$ also holds. Indeed, a stronger statement is
  proved in \cite[Section 5]{BarLiaRou19}.
  \label{rmk:SauBV}
\end{remark}

\begin{proof}[Proof of Theorem~\ref{the:measurezero}(\ref{infzero})] 
  Given a sequence $(r_n)_n$, let
  \begin{equation}\label{eq:S_n}
  	S_n(x,y)=\sum_{0\leq i,j< n} \1_{B(T_2^j y,r_n)}(T_1^i x).
  \end{equation}
  Note that $S_n(x,y)= 0$ implies $(x,y)\notin E_n$. We argue
  that the result follows if $\E(S_n)\to 0$. Indeed, since
  $S_n\geq 0$, there exists a subsequence $(n_k)_k$ s.t.\
  $S_{n_k}(x,y)\to 0$ a.s. Since $S_n(x,y)$ is integer-valued
  this means that a.s.\ $S_{n_k}(x,y)=0$ for all sufficiently
  large $k$. In particular, this means that
  $(x,y)\notin \liminf_{n} E_n$. Hence we compute the expectation
  of $S_n$. By the $T_1$-invariance of $\mu_1$ and
  $T_2$-invariance of $\mu_2$,
  \begin{align*}
  	\mathbbm{E}(S_n)&=\sum_{0\leq i,j< n} \iint \1_{B(T_2^j
  		y,r_n)}(T_1^i x) \,d\mu_1(x) d\mu_2(y)\\
  	&=\sum_{0\leq i,j< n} \int \mu_1 (B(y,r_n))\,d\mu_2(y)=n^2 \int \mu_1 (B(y,r_n))\,d\mu_2(y)\to 0.
  \end{align*}
  This concludes the proof.
  \end{proof}
  
 \begin{proof}[Proof of Theorem~\ref{the:measurezero}(\ref{supzero})] Given a sequence $(r_n)_n$ which decreases such that also $n\int\mu_1(B(y, r_n)) \, d\mu_2(y)$ is decreasing, let
  \begin{align*}
  	\tilde{S}_n (x, y)&:=\sum_{i, j\in [0, 2^{n+1})} \1_{B(T_2^jy, r_{2^n})}(T_1^i x),\\
  	\tilde{S} (x,y)&:=\sum_{n=0}^\infty \tilde{S}_n(x,y).
  \end{align*}
  We argue that the result follows if
  $\E(\tilde{S})<\infty$. Indeed, then $\tilde{S}(x,y)<\infty$
  a.s., which in turn means that a.s.\ $\tilde{S}_n(x,y)=0$ for all
  but finitely many $n\in \N$ since the $\tilde{S}_n$ are
  integer-valued. The definition of the $\tilde{S}_n$ along with
  the assumption $r_n\geq r_{n+1}$ means that
  $\tilde{S}_n(x,y)=0$ implies $(x,y)\notin E_{2^n + l}$ for all
  $l\in \{0,1,\dots, 2^{n+1}-2^n\}$. So if
  a.s.\ $\tilde{S}_n(x,y)=0$ for all sufficiently large $n$, then a.s.\ $(x,y)\notin \limsup_n E_n$.

 Hence we compute the expectation of $\tilde{S}$. Since $\tilde{S}_n\geq 0$ we have $\E(\tilde{S})=\sum_{n=1}^{\infty}\E(\tilde{S_n})$. By the $T_1$-invariance of $\mu_1$ and $T_2$-invariance of $\mu_2$,
  \begin{align*}
    \E (\tilde{S}) & = \sum_{n=0}^\infty \sum_{i, j\in [0, 2^{n+1})}\int\int
             \1_{B(T_2^jy, r_{2^n})}(T_1^ix) \,
             d\mu_1(x)d\mu_2(y)\\
           &= \sum_{n=0}^\infty\sum_{i, j\in [0,
             2^{n+1})}\int\mu_1 (B(T_2^jy,
             r_{2^n})) \, d\mu_2(y)\\
             &= \sum_{n=0}^\infty 2^{2(n+1)}\int
             \mu_1 (B(y, r_{2^n})) \, d\mu_2(y).
  \end{align*}
  Hence the assumption that
  $\sum_{n=0}^\infty 2^{2n}\int \mu_1 (B(y, r_{2^n})) \, d\mu_2(y)
  < \infty$ implies that $\E(\tilde{S})$ is bounded.
  Since $n\int\mu_1(B(y, r_n)) \, d\mu_2(y)$ is assumed decreasing, Lemma~\ref{lem:CauchyC} gives that this is equivalent to the
  assumption
  \[
    \sum_{n=0}^\infty n\int\mu_1(B(y, r_n)) \, d\mu_2(y) <\infty.
  \]
  This concludes the proof.
  \end{proof}
  
  \begin{proof}[Proof of Theorem~\ref{thm:supinf2}(\ref{thm:supinf2-1})]
  Suppose, without loss of
  generality, that the exponential mixing constants are the same
  for both systems. 
  
  Given a decreasing sequence $(r_n)_n$, for $x, y\in X$,
  define
  \[
    \hat S_n(x, y):= \sum_{i, j\in [0, 2^{n})} \1_{B(T_2^jy,
      r_{2^{n+1}})}(T_1^ix).
  \]
  The motivation for defining $\hat{S}_n$ along the subsequence $2^n$ will become clear later in the proof. Note that $\hat{S}_n$ is constructed so that if for some $n\in\N$ we have
  $\hat{S}_{n} (x,y) \geq 1$, then since $(r_n)_n$ is decreasing, also $\sum_{i,j\in [0,l)} \1_{B(T_2^jy, r_{l})}(T_1^ix)\geq 1$ for all $l\in [ 2^n, 2^{n+1}]$. Hence if $\hat{S}_n(x,y)\geq 1$ for all large $n$, then
  $(x,y) \in \liminf_n E_n$ and if $\hat{S}_n(x,y)\geq 1$ for all large $n$ is true a.s., then $(\mu_1\times\mu_2)(\liminf_n E_n)=1$. We argue that the result follows if $\E(\hat S_n)\to\infty$ and
  \begin{equation}\label{eq:SummableVariance}
  	\sum_{n=1}^\infty \E \left(\frac{\hat S_n}{\E(\hat S_n)}- 1
  	\right)^2 <\infty.
  \end{equation}
  The argument is standard and well-known but we include it for completeness. By the Markov inequality, since $\left(\frac{\hat S_n}{\E(\hat S_n)}- 1
  \right)^2\geq 0$, we get that for any $\delta>0$, 
  \begin{equation*}
  (\mu_1\times \mu_2)\left(\left(\frac{\hat S_n(x,y)}{\E(\hat S_n)}- 1\right)^2\geq \delta\right)\leq \frac{1}{\delta} \E \left(\frac{\hat S_n}{\E(\hat S_n)}- 1\right)^2.
  \end{equation*}
By \eqref{eq:SummableVariance}, the right hand side is summable and hence the left hand side is as well By the Borel--Cantelli Lemma, along with the fact that $\delta>0$ is arbitrary, we conclude that a.s.\ $\bigl| \frac{\hat S_n(x,y)}{\E(\hat S_n)}- 1\bigr|\geq \delta$ is true for at most finitely many $n\in\N$. In other words, a.s.\ for all $n$ sufficiently large we have $\frac{\hat S_n(x,y)}{\E(\hat S_n)}\in (1-\delta, 1+\delta)$. Since $\E(\hat S_n)\to\infty$ also a.s.\ $\hat S_n(x,y)\to\infty$ and in particular a.s.\ $\hat S_n(x,y)\geq 1$ for all $n$ sufficiently large which concludes the argument.  

We now prove $\E(\hat S_n) \to\infty$ and \eqref{eq:SummableVariance}. Using $T_1$-invariance of $\mu_1$ and $T_2$-invariance of $\mu_2$ we first compute
 \[\E(\hat S_n)= 2^{2n}\int\mu_1(B(y, r_{2^{n+1}})) \, d\mu_2(y)\to\infty. 
 \] 
by the assumption on the lower bound of shrinking rate of the integral.  
  
Next, note that the summands in \eqref{eq:SummableVariance} are
equal to $\frac{\E(\hat S_n^2)-\E(\hat S_n)^2}{\E(\hat
  S_n)^2}$. Since we already have $\E(\hat S_n)$ we proceed to
estimate $\E(\hat S_n^2)$. Because
  \[
    \E(\hat S_n^2) = \sum_{i_1, i_2, j_1, j_2 \in [0, 2^{n})}
    \int\int\1_{B(T_2^{j_1}y, r_{2^{n+1}})} (T_1^{i_1}x)\cdot
    \1_{B(T_2^{j_2}y, r_{2^{n+1}})} (T_1^{i_2}x) \,
    d\mu_1(x)d\mu_2(y),
  \]
  it suffices to assume that $i_1<i_2$, $j_1<j_2$ since for the three other combinations of inequalities, the upcoming estimates will give the same bound. For this case of indices we then
  split the sum into the four pairs arising from the cases where
  for some $c>0$ (to be chosen later) $i_2-i_1\le cn$,
  $i_2-i_1> cn$, $j_2-j_1\le cn$, $j_2-j_1> cn$.  
  
  For $i_2-i_1> cn$ and $j_2-j_1> cn$, i.e.\ \emph{the totally separated case},
  \begin{align*}
    & \sum_{\substack{i_1, i_2\in [0, 2^{n})\\ i_2-i_1> cn}} \
    \sum_{\substack{j_1, j_2\in [0, 2^{n})\\  j_2-j_1> cn}}
    \int\int\1_{B(T_2^{j_1}y, r_{2^{n+1}})} (T_1^{i_1}x)
    \1_{B(T_2^{j_2}y, r_{2^{n+1}})}(T_1^{i_2}x) \,
    d\mu_1(x)d\mu_2(y)\\ 
    &= \sum_{\substack{i_1, i_2\in [0, 2^{n})\\ i_2-i_1> cn}} \
    \sum_{\substack{j_1, j_2\in [0, 2^{n})\\  j_2-j_1> cn}}
    \int\int\1_{B(T_2^{j_1}y, r_{2^{n+1}})} (x)
    \1_{B(T_2^{j_2}y, r_{2^{n+1}})}(T_1^{i_2-i_1}x) \,
    d\mu_1(x)d\mu_2(y)\\ 
    &  \le \sum_{\substack{i_1, i_2\in [0, 2^{n})\\ i_2-i_1>
    cn}} \ \sum_{\substack{j_1, j_2\in [0, 2^{n})\\   j_2-j_1>
    cn}}  \biggl( \int\mu_1 (B(y, r_{2^{n+1}}))
    \mu_1 (B(T_2^{j_2-j_1}y, r_{2^{n+1}})) \, d\mu_2(y) + \\
    & \hspace{10.5cm} + 2 C e^{- n \theta c}\biggr)\\
    &   \le \sum_{\substack{i_1, i_2\in [0, 2^{n})\\
    i_2-i_1> cn}} \ \sum_{\substack{j_1, j_2\in [0, 2^{n})\\
    j_2-j_1> cn}}  \biggl( \biggl(\int\mu_1 (B(y, r_{2^{n+1}}))
    \, d\mu_2(y) \biggr)^2 + 4Ce^{-n\theta c}\biggr)\\
    & \le 2^{4n}\biggl( \biggl(\int\mu_1 (B(y,
      r_{2^{n+1}})) \, d\mu_2(y)\biggr)^2+ 4Ce^{-n\theta
      c}\biggr) = \E(\hat S_n)^2+ 2^{4n+2}Ce^{-n\theta c},
  \end{align*}
  where the first two inequalities follow from the mixing
  properties of the two systems, and in the second we use
  Lemma~\ref{lem:BV} or Remark~\ref{rmk:SauBV}.

  For $i_2-i_1\le cn$ and $j_2-j_1\le cn$, i.e.\ \emph{the totally non-separated case},
  \begin{align*}
    & \sum_{\substack{i_1, i_2\in [0, 2^{n})\\ 0\le i_2-i_1\le cn}}
    \ \sum_{\substack{j_1, j_2\in [0, 2^{n})\\  0\le j_2-j_1\le
    cn}} \int\int\1_{B(T_2^{j_1}y, r_{2^{n+1}})} (T_1^{i_1}x) \cdot
    \1_{B(T_2^{j_2}y, r_{2^{n+1}})} (T_1^{i_2}x) \,
    d\mu_1(x)d\mu_2(y)\\
    &\hspace{4.5cm} \le  \sum_{\substack{i_1, i_2\in [0, 2^{n})\\
    0\le i_2-i_1\le cn}} \ \sum_{\substack{j_1, j_2\in [0,
    2^{n})\\
    0\le j_2-j_1\le cn}} \int \mu_1 (B(T_2^{j_1}y, r_{2^{n+1}}))
    \, d\mu_2(y)\\
    &\hspace{4.5cm} =\quad (cn)^2 2^{2n}\int \mu_1 (B(y,
      r_{2^{n+1}})) \, d\mu_2(y)=  (cn)^2 \E(\hat S_n).
  \end{align*}

  For $i_2-i_1> cn$ and $j_2-j_1\le cn$, i.e.\ \emph{a half-separated case} (we can swap $T_1$ and $T_2$ to
  get the other half-separated case),
  \begin{align*}
    & \sum_{\substack{i_1, i_2\in [0, 2^{n})\\ i_2-i_1> cn}} \
    \sum_{\substack{j_1, j_2\in [0, 2^{n})\\  0\le j_2-j_1\le
    cn}} \int \int \1_{B(T_2^{j_1}y, r_{2^{n+1}})}
    (T_1^{i_1}x)\cdot \1_{B(T_2^{j_2}y, r_{2^{n+1}})}
    (T_1^{i_2}x) \, d\mu_1(x)d\mu_2(y)\\
    & \quad = \sum_{\substack{i_1, i_2, j_1, j_2 \in [0, 2^{n})\\
    i_2-i_1> cn\\ 0\le j_2-j_1\le cn}} \int\int\1_{B(T_2^{j_1}y,
    r_{2^{n+1}})}(x)\cdot \1_{B(T_2^{j_2}y,
    r_{2^{n+1}})}(T_1^{i_2-i_1}x) \, d\mu_1(x)d\mu_2(y)\\ 
    & \quad \le \sum_{\substack{i_1, i_2\in [0, 2^{n})\\ i_2-i_1>
    cn}} \ \sum_{\substack{j_1, j_2\in [0, 2^{n})\\  0\le
    j_2-j_1\le cn}}  \biggl( \int\mu_1 (B(T_2^{j_1}y,
    r_{2^{n+1}})) \mu_1 (B(T_2^{j_2}y,
    r_{2^{n+1}})) \, d\mu_2(y) + \\
    & \hspace{10.5cm} + 2 C e^{-n\theta c}\biggr)\\
    &\quad \le   \sum_{\substack{i_1, i_2\in [0, 2^{n})\\ i_2-i_1>
    cn}}  \ \sum_{\substack{j_1, j_2\in [0, 2^{n})\\  0\le
    j_2-j_1\le cn}} \biggl(\int \mu_1 (B (y, r_{2^{n+1}}))^2 \,
    d\mu_2(y) + 2 C e^{-n\theta c}\biggr)\\ 
&\quad \le cn2^{3n} \biggl(\int \mu_1 (B(y,
r_{2^{n+1}}))^2 \, d\mu_2(y) + 2 C e^{-n\theta c} \biggr),  
  \end{align*}
  where in the penultimate line we use the H\"older inequality and
  $T_2$-invariance of $\mu_2$.

  As mentioned earlier, we aim to show that
  $\E(\hat S_n^2) - \E(\hat S_n)^2$, when divided by
  $ \E(\hat S_n)^2$ are summable. Note that the definition of $\hat{S}_n$ along the subsequence $2^n$ becomes essential at this stage in the proof. Collecting our estimates from above we see that if we ignore multiplicative constants, the $\frac{\E(\hat S_n^2)-\E(\hat S_n)^2}{\E(\hat S_n)^2}$ are bounded by
  \begin{align*}
  \frac{2^{4n}e^{-n\theta c}}{\E(\hat S_n)^2}+\frac{n^2 \E(\hat S_n)}{\E(\hat S_n)^2}+ \frac{n2^{3n}\int \mu_1 (B(y, r_{2^{n+1}}))^2 \, d\mu_2(y)}{\E(\hat S_n)^2} + \frac{n2^{3n}e^{-n\theta c}}{\E(\hat S_n)^2}.
  \end{align*}
  The fourth summand is dominated by the first so it suffices to demonstrate that the first three are summable. For the first summand we have,
  \begin{equation*}
  	\frac{2^{4n}e^{-n\theta c}}{\E(\hat S_n)^2}=\frac{e^{-n\theta c}}{\bigl( \int\mu_1(B(y, r_{2^{n+1}})) \, d\mu_2(y) \bigr)^2}\leq \frac{2^{4n+4}e^{-n\theta c}}{(n^3(\log\log 2+ \log
  		n)^{1+\eps} (\log 2)^3 )^2}.
  \end{equation*}
  by the assumption that $\int\mu_1(B(y, r_n)) \, d\mu_2(y) \ge \frac{(\log n)^3(\log\log n)^{1+\eps}}{n^2}$. So choosing $c\ge 4\log 2/\theta$, this is summable. 
  
  Using the same assumption on the integral we get for the second summand that
  \begin{equation*}
  	\frac{n^2 \E(\hat S_n)}{\E(\hat S_n)^2}  =\frac{n^2}{ 2^{2n} \int\mu_1(B(y,r_{2^{n+1}})) \, d\mu_2(y)}\leq \frac{1}{n(\log\log 2+ \log n)^{1+\eps} (\log 2)^3 },
  \end{equation*}
  which is summable.
  
  For the third summand we apply Lemma~\ref{lem:tight} and the
  assumption \eqref{eq:extraconditionA} for $i = 1$ to obtain
  \begin{align*}
  	&\frac{n2^{3n}\int \mu_1 (B(y, r_{2^{n+1}}))^2 \, d\mu_2(y)}{\E(\hat S_n)^2} \\
  	&\leq  \frac{n2^{3n} \bigl( \int\mu_1(B(y, r_{2^{n+1}})) \,
  		d\mu_1(y) \bigr)^{\frac12} \int\mu_1(B(y, r_{2^{n+1}})) \,
  		d\mu_2(y)}{2^{4n} \biggl(\int\mu_1(B(y, r_{2^{n+1}})) \,
  		d\mu_2(y)\biggr)^2} \\
  	&= \frac{n \bigl( \int\mu_1(B(y, r_{2^{n+1}})) \, d\mu_1(y)
  		\bigr)^{\frac12}}{2^n \int\mu_1(B(y, r_{2^{n+1}})) \,
  		d\mu_2(y)}.\\
  		&\leq \frac{2}{n(\log\log 2 + \log n)^{1+\eps}(\log 2)^2}
  \end{align*}  
  which is also summable.
  
  As mentioned above, the other half-separated case is obtained
  by switching $T_1$ and $T_2$ and applying
  \eqref{eq:extraconditionA} for $i = 2$ instead of $i = 1$.
  This concludes the proof of part (\ref{thm:supinf2-1}).
\end{proof}

  \begin{proof}[Proof of Theorem~\ref{thm:supinf2}(\ref{thm:supinf2-2})] Given a sequence $(r_n)_n$, let again
  \begin{equation*}
  	S_n(x,y)=\sum_{0\leq i,j< n} \1_{B(T_2^j y,r_n)}(T_1^i x).
  \end{equation*}
  Note that by the assumption on the shrinking rate of the integral,
  \begin{equation*}
  	\E(S_n)=n^2 \int \mu_1 (B(y,r_n))\,d\mu_2(y)\geq (\log n)^2 h(n)\to\infty.
  \end{equation*}
	We argue that the result follows if $\E \Bigl(\frac{ S_n}{\E(S_n)}- 1
	\Bigr)^2\to 0$. Indeed, then there exists a subsequence $(n_k)_k$ s.t.\ $\Bigl(\frac{ S_{n_k}(x,y)}{\E(S_{n_k})}- 1
	\Bigr)^2\to 0\Rightarrow \frac{ S_{n_k}(x,y)}{\E(S_{n_k})}\to 1$ a.s. Since the denominator goes to infinity, so must $S_{n_k}(x,y)$ for $k\to\infty$ a.s. This implies that $(x,y)\in\limsup_n E_n$ a.s.
	
	To estimate the quantity $\E \Bigl(\frac{ S_n}{\E(S_n)}- 1 \Bigr)^2$ we follow the steps of part (\ref{thm:supinf2-1}). The difference to part (\ref{thm:supinf2-1}) is limited to the fact that we sum over the indices in $[0,n)$ instead of $[0,2^n)$ and the threshold for indices to be \enquote{separated} is given by $c\log n$ instead of $cn$. Repeating the same calculations as in part (\ref{thm:supinf2-1}), we obtain for part (\ref{thm:supinf2-2}) the estimate (again ignoring multiplicative constants)
	\begin{equation*}
		\frac{n^4 e^{-c\theta \log n}}{\E(S_n)^2} + \frac{(\log n)^2 \E(S_n)}{\E(S_n)^2} + \frac{ (\log n)n^3\int \mu_1 (B(y, r_n))^2 \, d\mu_2(y)}{\E(S_n)^2} + \frac{(\log n)n^3e^{-c\theta \log n} }{\E(S_n)^2} 
	\end{equation*} 
	on the quantity of interest. Again, the fourth summand is dominated by the first so it suffices to demonstrate that the first three summands vanish. 
	
	For the first summand we get,
	\begin{align*}
          \frac{n^4 e^{-c\theta \log n}}{\E(S_n)^2} =
          \frac{1}{n^{c\theta}\bigl (\int \mu_1(B(y,r_n))\,
          d\mu_2(y)\bigr )^2}\leq \frac{n^4}{n^{c\theta}(\log
          n)^4(h(n))^2}
	\end{align*}
	by the assumption that
        $\int\mu_1(B(y, r_n)) \, d\mu_2(y) \ge \frac{(\log n)^2
          h(n)}{n^2}$. So choosing $c\geq \frac{4}{\theta}$, this
        goes to zero.
	
	Using the same assumption on the integral we get for the
        second summand that
	\begin{equation*}
          \frac{(\log n)^2 \E(S_n)}{\E(S_n)^2} = \frac{(\log
            n)^2}{ n^2 \int \mu_1(B(y,r_n)) \, d\mu_2(y)} \leq
          \frac{1}{h(n)},
	\end{equation*}
	which goes to zero.
	
	For the third summand we apply Lemma~\ref{lem:tight} and
        the assumption \eqref{eq:extraconditionB} to obtain
	\begin{equation*}
          \frac{ (\log n)n^3\int \mu_1 (B(y, r_n))^2 \,
            d\mu_2(y)}{\E(S_n)^2} \leq \frac{(\log n)n^3\bigl
            (\int \mu_1 (B(y, r_n))^2 \, d\mu_1(y)\bigr
            )^{\frac12}}{n^4 \int \mu_1 (B(y, r_n))^2 \,
            d\mu_2(y)} \leq \frac{1}{h(n)},
	\end{equation*}
	which goes to zero. 
	
	This concludes the proof of part (\ref{thm:supinf2-2}).
\end{proof}

\subsection{Proofs for the doubling map}

For the proofs relating to the doubling map we will use the
notation $\mu^2:=(\mu\times\mu)$ to simplify expressions.

\begin{proof}[Proof of Theorem~\ref{thm:2xmod1Dicho}]
   
  The zero measure case follows from
  Theorem~\ref{the:measurezero}(\ref{supzero}).  We now
  consider the measure 1 case.

  Clearly,
  \begin{align*}
    G_n&:=\bigcup_{0\leq i\leq j< n} \{(x,y): d(T^i x,T^j
         y)\leq r_n \}\\ 
       &\subset \bigcup_{0\leq i,j< n} \{(x,y): d(T^i
         x,T^j y)\leq r_n \}=E_n, 
  \end{align*}
  implying that $\limsup_{n}G_n\subset \limsup_{n}E_n$. In the
  following we will construct a subset of $\limsup_{n}G_n$ which
  is easier to work with. Consider now
  \begin{equation*}
    H_n:=\bigcup_{0\leq i\leq j< n} \{(x,y): d(T^i x,T^j
    y)\leq r_j \}.
  \end{equation*}
  Write
  \begin{align*}
    &C_{i,j}:=\{(x,y): d(T^i x,T^j y)\leq r_j\}\\
    &B_{i,j,n}:=\{(x,y): d(T^i x,T^j y)\leq r_n\}
  \end{align*}
  such that
  \begin{align*}
    &H_n=\bigcup_{0\leq i\leq j< n} C_{i,j}\\
    &G_n=\bigcup_{0\leq i\leq j< n} B_{i,j,n}.
  \end{align*}
  For a given $j\leq n$, $C_{i,j}=B_{i,j,j}\subset G_j$. Hence
  \begin{equation*}
    H_n\subset \bigcup_{j=1}^n G_j. 
  \end{equation*}
  We next want to restructure
  $J_n:=\{(i,j): 1\leq i\leq j\leq n\}$ from being an array to a
  sequence. We do this by introducing the ordering $\preceq$
  whereby $(i_1,j_1)\preceq (i_2,j_2)$ if $j_2>j_1$ or if
  $j_2=j_1$ and $i_2\geq i_1$. Using this ordering we may
  reenumerate the $C_{i,j}$'s chronologically from 1 up to
  $|J_n|=\frac{n(n+1)}{2}$, i.e.\ for each $(i,j)\in J_n$,
  $C_{i,j}=:A_l$ for some $l\in
  \{1,\dots,\frac{n(n+1)}{2}\}$. Note that this restructuring may
  equally well be done for the infinite array
  $J:=\{(i,j): 1\leq i\leq j\}$ giving rise to an infinite
  sequence. This restructuring allows us to consider the
  $\limsup$ set of the $C_{i,j}=A_l$'s, i.e.
  \begin{equation*}
    \limsup_{l}A_l=\bigcap_{l=1}^{\infty}\bigcup_{k=l}^{\infty}
    A_k.
  \end{equation*}
  Suppose $(x,y)\in \limsup_{l} A_l$. This means that
  $(x,y)\in A_l$ for infinitely many $l$, or similarly,
  $(x,y)\in C_{i,j}$ for infinitely many distinct pairs
  $(i,j)\in J$. In particular, since there are for each $j$ only
  finitely many options for $i$, there exists pairs $(i,j)\in J$
  with arbitrarily large value of $j$ for which
  $(x,y)\in C_{i,j}$ . Since $C_{i,j}\subset G_j$ we have that
  $(x,y)\in G_j$ for infinitely many $j$. In other words,
  $(x,y)\in \limsup_{n} G_n$.
		
  We conclude that
  \begin{equation*}
    \limsup_{l}A_l\subset \limsup_{n} G_n\subset
    \limsup_{n} E_n  
  \end{equation*}
  and consequently
  \begin{equation*}
    \mu^2(\limsup_{l} A_l) \leq \mu^2\left(\limsup_{n} E_n\right).  
  \end{equation*}
  We will prove that $\mu^2\left(\limsup_{l} A_l\right) =1$, thereby
  concluding $\mu^2\left(\limsup_{n} E_n\right)=1$.
		
  Assume that 
  \begin{equation}\label{divergence}
    \sum_{n=1}^{\infty} nr_n=\infty.
  \end{equation}
  Extending $\1_{B(0,r_j)}$ to a periodic function on $\R$, for
  any given $l\in\N$ we have
  \begin{equation*}
    \mu^2(A_l)=\mu^2(C_{i,j})=\int \1_{C_{i,j}}d\mu^2=\iint
    \1_{B(0,r_j)}(2^i x-2^j y)\,dx\,dy
  \end{equation*}
  for some $(i,j)\in J$. We may write the periodic extension of
  $\1_{B(0,r_j)}$ via its Fourier series,
  \begin{equation*}
    \1_{B(0,r_j)}(z)=\sum_{k\in\Z} c_{j,k}e^{2\pi i k z}
  \end{equation*}
  which means that 
  \[
    \mu^2(A_l) = \sum_{k\in\Z} c_{j,k}\iint e^{2\pi i k (2^i x- 2^j
                y)} \,d x\,dy=c_{j,0}=\int
                \1_{B(0,r_j)}d\mu^2=2r_j 
  \]
  since the integrals are zero unless $k=0$ in which case it is
  1. From this we can compute the partial sums of the
  $\mu^2(A_l)$, namely, we write any given $M\in \N$ as $M=k_n+a$
  for some $n\in\N$ where $k_n=\sum_{i=1}^n i$ and $1\leq a<n+1$,
  thereby obtaining
  \begin{align*}
    \sum_{l=1}^{M}\mu^2(A_l)
    &=\sum_{l=1}^{k_n}\mu^2(A_l)+\sum_{l=k_n+1}^{M}\mu^2(A_l)
      =\sum_{{1\leq i\leq j\leq n}}
      \mu^2(C_{i,j})+\sum_{i=1}^{a}\mu^2(C_{i,n+1})\\ 
    &=\sum_{j=1}^{n} j2r_j+
      a2r_{n+1}\stackrel{n\to\infty}{\longrightarrow}\infty 
  \end{align*}
  where the divergence holds by assumption
  \eqref{divergence}. Clearly $M\to\infty$ for $n\to\infty$
  implying the divergence of the first sum as well. We aim to
  employ the following consequence of the Erd\H{o}s--R\'{e}nyi
  version of the Borel--Cantelli lemma \cite{ErdRen59} .
  \begin{lemma}\label{Chung-Erdös}
    Let $(Y,\mathcal{B}, \nu)$ denote a probability space. If
    $A_l\in \mathcal{B}$ are sets such that
    \begin{equation}\label{eq:divergentsum}
      \sum_{l=1}^{\infty} \nu(A_l)=\infty
    \end{equation}
    and
    \begin{equation}
      \label{eq:harman}
      \liminf_{N \to \infty}
      \frac{ \displaystyle \sum_{1 \leq l_1 < l_2 \leq N} \bigl(
        \nu (A_{l_1} \cap A_{l_2}) - \nu (A_{l_1}) \nu (A_{l_2})
        \bigr)}{\displaystyle \Biggl( \sum_{l=1}^N \nu (A_l)
        \Biggr)^2} \leq 0,
    \end{equation}
    then $\nu (\limsup_l A_l) = 1$.
  \end{lemma}
  Since we already showed that property \eqref{eq:divergentsum}
  is satisfied, we are left with showing that \eqref{eq:harman}
  also holds. For this purpose we again invoke Fourier series. In
  the following, sets $A_{l_1}, A_{l_2}$ will be identified with
  the sets $C_{i_1,j_1}, C_{i_2,j_2}$ respectively when indexed
  via an array. Since we sum over $l_1<l_2$ in condition
  \eqref{eq:harman} we may assume without loss of generality that
  $j_2\geq j_1$. We have that
  \begin{align*}
    \mu^2(A_{l_1}\cap A_{l_2})
    &=\iint
      \1_{B(0,r_{j_1})}(2^{i_1}x-2^{j_1}y)
      \1_{B(0,r_{j_2})}(2^{i_2}x-2^{j_2}y)\,dx\,dy\\
    &=\sum_{k_1,k_2\in\Z} c_{j_1,k_1}c_{j_2,k_2}\iint e^{2\pi i
      \left(k_1(2^{i_1}x-2^{j_1}y) +
      k_2(2^{i_2}x-2^{j_2}y)\right)}\,dx\,dy\\
    &=\sum_{k_1,k_2\in\Z} c_{j_1,k_1}c_{j_2,k_2} \int e^{2\pi i
      (k_1 2^{i_1}+k_2 2^{i_2})x}\,dx\int e^{-2\pi i (k_1
      2^{j_1}+k_2 2^{j_2})y}\,dy.
  \end{align*}
  We see that the integrals are only non-zero if
  \begin{equation*}
    \begin{cases}
      k_1 2^{i_1}+k_2 2^{i_2} = 0\\
      k_1 2^{j_1}+k_2 2^{j_2}= 0
    \end{cases}
    \Rightarrow\quad
    \begin{cases}
      k_1=-k_2 2^{i_2-i_1}\\
      k_1=-k_2 2^{j_2-j_1}.
    \end{cases}
  \end{equation*}
  The last two equations can only be true if $p:=j_2-j_1=i_2-i_1$
  or $k_1=k_2=0$. Suppose that $j_2-j_1=i_2-i_1$. In this case
  the sum over $k_1\in\Z$ can be replaced by summing over the
  values $-k_2 2^{p}$. Furthermore, the central coefficients
  satisfy $c_{j_1,0}c_{j_2,0}=\mu(A_{l_1})\mu(A_{l_2})$ which we
  may subtract
  to the left hand side to obtain
  \begin{align*}
    \mu^2(A_{l_1}\cap
    A_{l_2})-\mu^2(A_{l_1})\mu^2(A_{l_2})
    &=\sum_{k_2\in\Z\backslash\{0\}}
      c_{j_1,-k_2
      2^{p}}c_{j_2,k_2}\\ 
    &\leq \sum_{k_2\in\Z\backslash\{0\}} |c_{j_1,-k_2
      2^{p}}c_{j_2,k_2}|.
  \end{align*} 
  We have two upper bounds for the coefficients, namely the
  inverse of the index over which the coefficient is summed as
  well as the integral of the function that the Fourier series
  represents. More precisely,
  \begin{equation*}
    |c_{j_1,-k_2 2^{p}}c_{j_2,k_2}|\leq
    \begin{cases}
      \frac{1}{|k_2| 2^p}\frac{1}{|k_2|}=\frac{1}{k_2^2 2^p}\\
      \int\1_{B(0,r_{j_1})}\,d\mu^2
      \int\1_{B(0,r_{j_2})}\,d\mu^2=4r_{j_1}r_{j_2} 
    \end{cases}.
  \end{equation*} 
  To optimize our upper bound we compute
  \begin{equation*}
    4r_{j_1}r_{j_2}\leq \frac{1}{k_2^2 2^p} \enskip
    \Rightarrow\enskip k_2\leq
    (4r_{j_1}r_{j_2})^{-\frac12}2^{-\frac{p}{2}}
  \end{equation*}
  and split up our sum accordingly, that is, for
  $S_{j_1,j_2}:=(4r_{j_1}r_{j_2})^{-\frac12}2^{-\frac{p}{2}}$  we
  have
  \begin{align*}
    \sum_{k_2\in\Z\backslash\{0\}} |c_{j_1,-k_2 2^{p}}c_{j_2,k_2}|
    &=\sum_{k_2=1}^{\lfloor S_{j_1,j_2} \rfloor } |c_{j_1,-k_2
      2^{p}}c_{j_2,k_2}| + \sum_{k_2=S_{j_1,j_2}}^{\infty}
      |c_{j_1,-k_2 2^{p}}c_{j_2,k_2}|\\ 
    &\leq 2 \left (\sum_{k_2=1}^{\lfloor S_{j_1,j_2} \rfloor }
      4r_{j_1}r_{j_2} + \sum_{k_2=S_{j_1,j_2}}^{\infty}
      \frac{1}{k_2^2 2^p}\right )\\ 
    &\leq
      8r_{j_1}r_{j_2}(4r_{j_1}r_{j_2})^{-\frac12}2^{-\frac{p}{2}}+\frac{2}{S_{j_1,j_2}
      2^p}  \\ 
    &\ll (r_{j_1}r_{j_2})^{\frac12}2^{-\frac{p}{2}}. 
  \end{align*}
  In total we conclude that 
  \begin{equation*}
    \mu^2(A_{l_1}\cap A_{l_2})-\mu^2(A_{l_1})\mu^2(A_{l_2})
    \begin{cases}
      = 0 & \text{ if } j_2-j_1\neq i_2-i_1,\\
      \ll (r_{j_1}r_{j_2})^{\frac12}2^{-\frac{j_2-j_1}{2}} & \text{ if } j_2-j_1= i_2-i_1.
    \end{cases}
  \end{equation*}
  With this estimate in hand we are ready to prove property
  \eqref{eq:harman} of Lemma~\ref{Chung-Erdös}. We have
  \begin{align*}
    \displaystyle \sum_{1 \leq l_1 <
    l_2 \leq N} \bigl(\mu^2(A_{l_1} \cap A_{l_2}) - \mu^2(A_{l_1})\mu^2
    (A_{l_2}) \bigr) &\ll  \displaystyle \sum_{1 \leq l_1 <
                       l_2 \leq N} (r_{j_1}r_{j_2})^{\frac12}2^{-\frac{j_2-j_1}{2}}\\
                     &\leq \displaystyle \sum_{1 \leq l_1 \leq
                       l_2 \leq N} r_{j_1} 2^{-\frac{j_2-j_1}{2}}				
  \end{align*}
  since $r_j\geq r_{j+1}$ by assumption. Recall that the sets
  $A_{l_1}, A_{l_2}$ correspond to the sets
  $C_{i_1,j_1}, C_{i_2,j_2}$, hence we need to rewrite the sum
  over $1\leq l_1\leq l_2\leq N$ as a sum over the appropriate
  index set of the $i_1,i_2,j_1,j_2$. Since the condition in
  \eqref{eq:harman} is given via the $\liminf$, it is sufficient
  if we can prove that the fraction vanishes along a
  subsequence. We will do this for the subsequence $N=k_n$ where
  $k_n=\sum_{i=1}^n i$. Define the sets
  \begin{align*}
    I_n&:=\left \{ (i_1,i_2,j_1,j_2)\in \N^4:
         \begin{array}{cc}
           i_1\leq j_1\leq n   &  j_2-j_1=i_2-i_1 \\  i_2\leq
           j_2\leq n &  j_1\leq j_2
         \end{array} \right \}.
  \end{align*}
  So we have
  \begin{equation*}
    \displaystyle \sum_{1 \leq l_1 \leq
      l_2 \leq N} r_{j_1} 2^{-\frac{j_2-j_1}{2}}= \displaystyle
    \sum_{I_n} r_{j_1} 2^{-\frac{j_2-j_1}{2}}.
  \end{equation*}
  First we note that the sum on the right hand side is
  independent of $i_1, i_2$. For fixed $j_1\leq j_2\leq n$ there
  are exactly $j_1$ pairs $(i_1,i_2)$ such that
  $(i_1,i_2,j_1,j_2)\in I_n$, namely
  \begin{equation*}
    \begin{array}{c}
      (1,1+(j_2-j_1)) \\ (2,2+(j_2-j_1)) \\ \vdots \\ (j_1,j_1+(j_2-j_1)).
    \end{array}
  \end{equation*} 
  Hence we may write 
  \begin{align*}
    \displaystyle \sum_{I_n} r_{j_1} 2^{-\frac{j_2-j_1}{2}}
    &=\displaystyle \sum_{1\leq j_1\leq j_2\leq n}
      j_1 r_{j_1} 2^{-\frac{j_2-j_1}{2}}=\sum_{j_1=1}^{n} j_1
      r_{j_1} \sum_{j_2=j_1}^{n} 2^{-\frac{j_2-j_1}{2}}\\
    &\leq  \sum_{j_1=1}^{n} j_1 r_{j_1} \sum_{m=0}^{\infty}
      2^{-\frac{m}{2}}=C_1 \sum_{j_1=1}^{n} j_1 r_{j_1}.
  \end{align*}
  This gives \eqref{eq:harman} since then
  \begin{align*}
    \lim_{n \to \infty} \frac{ \displaystyle \sum_{1 \leq l_1 <
    l_2 \leq k_n} \bigl(\mu^2(A_{l_1} \cap A_{l_2}) -\mu^2(A_{l_1})\mu^2
    (A_{l_2}) \bigr)}{\displaystyle \Biggl( \sum_{l=1}^N \mu^2 (A_l)
    \Biggr)^2} &\leq \lim_{n \to \infty} \frac{C_1 \sum_{j=1}^{n} j
                 r_{j}}{(\sum_{j=1}^{n} j r_{j})^2}\\ 
               &= \lim_{n \to \infty}\frac{C_1 }{\sum_{j=1}^{n} j r_{j}}=0
  \end{align*}
  since $\sum_{j=1}^{n} j r_{j}\to\infty$ for $n\to\infty$ by
  assumption. We conclude that under this assumption
  $\mu^2(\limsup_{n} E_n)=1$.
\end{proof}

\begin{proof}[Proof of Theorem~\ref{the:2xliminf}] Part
  (\ref{2xliminf0}) is an immediate consequence of
  Theorem~\ref{the:measurezero}(\ref{infzero}).
	
  Proof of part (\ref{2xliminf1}).
  For a decreasing sequence $(r_n)_n$, we again let
  \begin{equation*}
    \hat{S}_n(x,y)=\sum_{i,j\in [0, 2^n)}  \1_{B(T^j
      y,r_{2^{n+1}})}(T^i x).
  \end{equation*}
  The structure of the proof is the same as that of the proof of
  Theorem~\ref{thm:supinf2}(\ref{thm:supinf2-1}). As
  explained in detail there, the result follows if we can prove
  that
  \begin{equation*}
    \sum_{n=1}^{\infty}
    \frac{\E(\hat{S}_n^2)-\E(\hat{S}_n)^2}{\E(\hat{S}_n)^2}<\infty.
  \end{equation*}
  To this end, we first compute the expectation of $\hat{S}_n$. By $T$-invariance of $\mu$ we get
  \begin{align}\label{eq:liminfexpt}
    \E(\hat{S}_n)
           &=\sum_{i,j\in [0, 2^n)} \int \mu(B(T^j
             y,r_{2^{n+1}})) \,d\mu(y)\nonumber\\ 
           &=\sum_{i,j\in [0, 2^n)} 2r_{2^{n+1}}=2^{2n+1} r_{2^{n+1}}.
  \end{align} 
  Note that condition \eqref{eq:liminf1} implies $\E(\hat{S}_n)\to \infty $.
  We are left with computing $\E(\hat{S}_n^2)$. We have
  \begin{align*}
    \hat{S}_n(x,y)^2
    &=\sum_{i_1,i_2,j_1,j_2\in [0,2^n)}
      \1_{B(T^{j_1}y,r_{2^{n+1}})}(T^{i_1}
      x)\1_{B(T^{j_2}y,r_{2^{n+1}})}(T^{i_2} x)\\ 
    &=\sum_{i_1,i_2,j_1,j_2\in [0,2^n)}
      \1_{B(0,r_{2^{n+1}})}(T^{i_1}
      x-T^{j_1}y)\1_{B(0,r_{2^{n+1}})}(T^{i_2} x-T^{j_2}y)\\ 
    &=\sum_{i_1,i_2,j_1,j_2\in [0,2^n)}
      \1_{B(0,r_{2^{n+1}})}(2^{i_1}
      x-2^{j_1}y)\1_{B(0,r_{2^{n+1}})}(2^{i_2} x-2^{j_2}y). 
  \end{align*}
  Taking the expectation then gives
  \begin{equation*}
    \E(\hat{S}_n^2)=\sum_{\substack{i_1,i_2\in [0,2^n) \\
        j_1,j_2\in [0,2^n)}} \iint
    \1_{B(0,r_{2^{n+1}})}(2^{i_1}
    x-2^{j_1}y)\1_{B(0,r_{2^{n+1}})}(2^{i_2} x-2^{j_2}y)\,dx\,dy.
  \end{equation*}
  We now proceed as in the proof of Theorem~\ref{thm:2xmod1Dicho}
  by writing $\1_{B(0,r_{2^{n+1}})}$ via its Fourier series. For given $i_1,i_2,j_1,j_2\in [1,2^n]$ we obtain 
  \begin{align*}
    \iint& \1_{B(0,r_{2^{n+1}})}(2^{i_1} x-2^{j_1}y)\1_{B(0,r_{2^{n+1}})}(2^{i_2} x-2^{j_2}y)\,dx\,dy\\
         &=\sum_{k_1,k_2\in\Z} c_{j_1,k_1}c_{j_2,k_2}\iint e^{2\pi i \bigl(k_1(2^{i_1} x-2^{j_1}y)+k_2(2^{i_2} x-2^{j_2}y)\bigr)}\,dx\,dy\\
         &= \sum_{k_1,k_2\in\Z} c_{j_1,k_1}c_{j_2,k_2}\int
           e^{2\pi i (k_1 2^{i_1} + k_2 2^{i_2})x}\,dx \int
           e^{2\pi i (k_1 2^{j_1}+k_2 2^{j_2})y}\,dy.
             \end{align*}
  As in the proof of Theorem~\ref{thm:2xmod1Dicho} we see that
  the integrals are only non-zero if
  \begin{equation*}
    \begin{cases}
      k_1 2^{i_1}+k_2 2^{i_2} = 0\\
      k_1 2^{j_1}+k_2 2^{j_2}= 0
    \end{cases}
    \Rightarrow\quad
    \begin{cases}
      k_1=-k_2 2^{i_2-i_1}\\
      k_1=-k_2 2^{j_2-j_1},
    \end{cases}
  \end{equation*}
  that is, when $k_1=k_2=0$ or when $p:=j_2-j_1=i_2-i_1$. Suppose
  $j_2-j_1=i_2-i_1$. In this proof we need to consider the two
  cases $p\geq 0$ and $p< 0$. For $p\geq 0$ we use the equation
  above for $k_1$, i.e.\ we sum over $k_1=-k_2 2^p$. For $p<0$ we
  instead replace the sum over $k_2$ by
  $k_2=-k_1 2^{-p}=-k_1 2^{|p|}$. However, referring again to
  the proof of Theorem~\ref{thm:2xmod1Dicho}, we see that the
  final estimate of the Fourier coefficients does not depend on
  $k_1$ or $k_2$, hence we can use the same bound. In total we
  get
  \begin{align*}
    \sum_{k_1,k_2\in\Z} c_{j_1,k_1}c_{j_2,k_2}
    &\int e^{2\pi i (k_1 2^{i_1} + k_2 2^{i_2})x}\,dx \int e^{2\pi i (k_1 2^{j_1}+k_2 2^{j_2})y}\,dy\\
    &\leq c_{j_1,0}c_{j_2,0}+\begin{cases}0 & \text{ if } j_2-j_1\neq i_2-i_1,\\
      r_{2^{n+1}}2^{-\frac{|j_2-j_1|}{2}} & \text{ if } j_2-j_1= i_2-i_1.
    \end{cases}
  \end{align*}
  Define the index set
  \begin{equation*}
    \tilde{I}_n:=\left \{ (i_1,i_2,j_1,j_2)\in \N^4:
      i_1,i_2, j_1,j_2\leq n  \text{ , }   j_2-j_1=i_2-i_1 \right \}.
  \end{equation*}
  The desired expectation can now be estimated through the
  expression
  \begin{align*}
    \E(\hat{S}_n)^2&\leq\sum_{i_1,i_2,j_1,j_2\in [0,2^n)}
                    4r_{2^{n+1}}^2+\sum_{i_1,i_2,j_1,j_2 \in
                    \tilde{I}_{2^n}}r_{2^{n+1}}
                    2^{-\frac{|j_2-j_1|}{2}}\\ 
                  &=2^{4n+2}r_{2^{n+1}}^2+ r_{2^{n+1}}
                    \sum_{i_1,i_2,j_1,j_2 \in
                    \tilde{I}_{2^n}}2^{-\frac{|j_2-j_1|}{2}}.
  \end{align*}
  We focus on the remaining sum. Since it does not depend on
  $i_1,i_2$ we consider the number of pairs $(i_1,i_2)$ for a
  given pair $(j_1,j_2)$, for which $ j_2-j_1=
  i_2-i_1$. Certainly this can be bounded from above by
  $2^n$. This means that
  \begin{align*}
    \E(\hat{S}_n)^2&\leq 2^{4n+2}r_{2^{n+1}}^2+ 2^n r_{2^{n+1}}\sum_{j_1,j_2\in[0,2^n)} 2^{-\frac{|j_2-j_1|}{2}}\\
                  &\leq 2^{4n+2}r_{2^{n+1}}^2+ 2^n r_{2^{n+1}} 2\sum_{j_1=0}^{2^n-1}\sum_{j_2=j_1}^{2^n} 2^{-\frac{j_2-j_1}{2}}\\
                  &= 2^{4n+2}r_{2^{n+1}}^2+ 2^n r_{2^{n+1}} 2\sum_{j_1=0}^{2^n-1}\sum_{m=0}^{\infty} 2^{-\frac{m}{2}}\\
                  &\leq 2^{4n+2}r_{2^{n+1}}^2+ 2^n r_{2^{n+1}} 2\sum_{j_1=0}^{2^n-1} C_1\\
                  &\leq 2^{4n+2}r_{2^{n+1}}^2+ C_1 2^{2n+1} r_{2^{n+1}},
  \end{align*}
  for some $C_1 >0$.
  
  Using this estimate along with \eqref{eq:liminfexpt} we can now compute
  \begin{align}\label{eq:propconclusion}
    \frac{\E(\hat{S}_n^2)-\E(\hat{S}_n)^2}{\E(\hat{S}_n)^2}
    &\leq
      \frac{2^{4n+2}r_{2^{n+1}}^2+
      C_1
      2^{2n+1} r_{2^{n+1}}-(2^{2n+1} r_{2^{n+1}})^2}{(2^{2n+1} r_{2^{n+1}})^2}\\
    &=\frac{C_1}{2^{2n+1}r_{2^{n+1}}}.\nonumber
  \end{align}
  Hence 
  \begin{equation*}
    \sum_{n=0}^{\infty}
    \frac{C_1}{2^{2n+1}r_{2^{n+1}}} = C_1
    \sum_{n=1}^{\infty}
    \frac{1}{2^{2n}r_{2^{n}}}<\infty 
  \end{equation*}
  by \eqref{eq:liminf1}. We conclude that $\mu^2(\liminf_{n} E_n)=1$.
\end{proof}

\section{Proofs for Rotations}

Parts (\ref{thm:rotinfzero}) and (\ref{thm:rotsupzero}) of Theorem~\ref{thm:rotation} follow immediately from Theorem~\ref{the:measurezero} exploiting that $\mu_1$ and $\mu_2$ are Lebesgue measure. The proof of parts (\ref{thm:rotinfone}) and (\ref{thm:rotsupone}), given below, is similar to that of Theorem~\ref{thm:supinf2}.

\begin{proof}[Proof of Theorem~\ref{thm:rotation}(\ref{thm:rotinfone})]
Given a decreasing sequence $(r_n)_n$, set again
  \[
    \hat S_n(x, y):= \sum_{i, j\in [0, 2^{n})} \1_{B(T_2^jy,
      r_{2^{n+1}})}(T_1^ix).
  \]
  As argued in Theorem~\ref{thm:supinf2}(\ref{thm:supinf2-1}), it is sufficient to show that
  $\sum_{n=1}^\infty \frac{\E(\hat S_n^2)-\E(\hat S_n)^2}{\E(\hat S_n)^2} <\infty$.  As in the doubling map case, we can write
  $\mathbb{E} (\hat{S}_n) = 2^{2n+1} r_{2^{n+1}}$.

  We begin by estimating $\E(\hat{S}_n^2)$ which we may also write as
  \begin{align*}
    & \sum_{i_1,i_2,j_1,j_2\in [0, 2^{n}) } \iint\1_{B(T_2^{j_1}y,
    r_{2^{n+1}})}(T_1^{i_1}x)\cdot \1_{B(T_2^{j_2}y,
    r_{2^{n+1}})}(T_1^{i_2}x)~d\mu_1(x)d\mu_2(y).
  \end{align*}

  Since only one the systems is mixing we distinguish just
  between two cases, namely the separated case $i_2-i_1>cn$ and
  the non-separated case $i_2-i_1\leq cn$ with
  $c>4\log(2)/\theta$. For the separated case we use exponential
  mixing of $T_1$ with respect to Lebesgue measure $\mu_1$ to
  obtain
  \begin{align*}
    & \sum_{j_1,j_2 \in [0, 2^{n})} \sum_{\substack{i_1, i_2\in [0, 2^{n})\\ i_2-i_1>cn}} \
    \iint\1_{B(T_2^{j_1}y,
    r_{2^{n+1}})}(T_1^{i_1}x)\cdot \1_{B(T_2^{j_2}y,
    r_{2^{n+1}})}(T_1^{i_2}x)~d\mu_1(x)d\mu_2(y)\\
    &\quad \leq \sum_{j_1,j_2 \in [0, 2^{n})}\sum_{\substack{i_1, i_2\in [0, 2^{n})\\
    i_2-i_1>cn}} \
    \int \Bigl(\mu_1(B(T_2^{j_1}y,
    r_{2^{n+1}})\mu_1(B(T_2^{j_2}y,
    r_{2^{n+1}}) + \\
    & \hspace{9cm} +2C \mathrm{e}^{-n\theta c}\Bigr)d\mu_2(y) \\
    & \quad \leq 2^{4n} \cdot \bigl(4r^2_{2^{n+1}}+2C \mathrm{e}^{-n\theta c}\bigr).
  \end{align*}
  To investigate the non-separated case $i_2-i_1\leq cn$ we write 
  \begin{align*}
    & \sum_{j_1,j_2 \in [0, 2^{n})} \sum_{\substack{i_1, i_2\in [0, 2^{n})\\ i_2-i_1\leq cn}} \
    \iint\1_{B(T_2^{j_1}y,
    r_{2^{n+1}})}(T_1^{i_1}x) \1_{B(T_2^{j_2}y,
    r_{2^{n+1}})}(T_1^{i_2}x)~d\mu_1(x)d\mu_2(y)\\
    &\quad = \sum_{\substack{i_1, i_2\in [0, 2^{n})\\ i_2-i_1\leq cn}}\sum_{j_1,j_2\in [0, 2^{n})}\
   \iint\1_{B(T_1^{i_1}x,
    r_{2^{n+1}})}(y) \cdot \\
    & \hspace{6cm} \cdot \1_{B(T_1^{i_2}x,
    r_{2^{n+1}})}(T_2^{j_2-j_1}y)~d\mu_2(y)d\mu_1(x)
  \end{align*}
and we note that it is enough to prove part (\ref{thm:rotinfone}) under the assumption that
  \[
    r_n = \frac{(\log n)^2(\log\log n)^{1+\eps}}{n^2}.
  \]
This will allow the application of the following claim.
  
  \begin{claim}\label{claim:Dioph}
    Suppose $\alpha$ admits a function
    $\Psi_{\alpha} \colon \N\to\R$ such that
    $\|j\alpha\| \geq \Psi_{\alpha}(n)$ for all $j\in [1,n]$ and
    that $r_n = o( \Psi_{\alpha}(n))$.

    Then, for all large enough $n$, if $y$ and $z$ are two given
    points, there is at most one $j\in [1,n]$ such that
    $T_2^j y \in B(z, r_n)$.
  \end{claim}

\begin{proof}
  The claim follows if we can show for all large enough $n$ that
  it is not possible for distinct $j_1, j_2\in [0, n)$ to have
  $d(T_2^{j_1}y, T_2^{j_2}y) < 2 r_n$.  So it is sufficient to
  show that
  $d(T_2^j0, 0) \ge \Psi_{\alpha}(n)$
  for all $j\in [1,n]$.  Another way of writing this is
  $\|j\alpha\| \ge \Psi_{\alpha}(n)$
  for $j\in [1,n]$, which is guaranteed by the Diophantine assumption
  on $\alpha$.
\end{proof}
Set $\Psi_{\alpha}(n)=c(\alpha) (\log n)^2 (\log\log(n))^{1+2\eps}/ n^{2}$. By the assumptions in part (\ref{thm:rotinfone}), we see that $r_n = o( \Psi_{\alpha}(n))$ and since $\Psi_{\alpha}(n)$ is decreasing also $\|j\alpha\|\geq \Psi_{\alpha}(n)$ for all $j\in [1,n]$. Thus $\Psi_{\alpha}(n)$ satisfies the requirements of the claim.

By the claim, for each each fixed $x$, $y$, $i_1$, $i_2$ and $j_1$ there can be at most one $j_2=j_2(x,y,i_1,i_2,j_1)\in [0, 2^{n})$ such that $\1_{B(T_1^{i_1}x,
    r_{2^{n+1}})}(y)\cdot \1_{B(T_1^{i_2}x,
    r_{2^{n+1}})}(T_2^{j_2-j_1}y)\neq 0$. We then write $\tilde{j}=\tilde{j}(x,y,i_1,i_2,j_1)\coloneqq j_2-j_1$ and our sum from above as
\[
  \sum_{\substack{i_1, i_2\in [0, 2^{n})\\ i_2-i_1\leq cn}}\sum_{j_1\in [0, 2^{n})} \
   \iint\1_{B(T_1^{i_1}x,
    r_{2^{n+1}})}(y)\cdot \1_{B(T_1^{i_2}x,
    r_{2^{n+1}})}(T_2^{\tilde{j}}y)~d\mu_2(y)d\mu_1(x)
\]
noting that $\tilde{j}$ may not exist in which case the integrad above is 0. All together we get
\begin{align*}
    & \sum_{j_1,j_2 \in [0, 2^{n})} \sum_{\substack{i_1, i_2\in [0, 2^{n})\\ i_2-i_1\leq cn}} \
    \iint\1_{B(T_2^{j_1}y,
    r_{2^{n+1}})}(T_1^{i_1}x)\cdot \1_{B(T_2^{j_2}y,
    r_{2^{n+1}})}(T_1^{i_2}x)~d\mu_1(x)d\mu_2(y)\\
    &\quad = \sum_{\substack{i_1, i_2\in [0, 2^{n})\\ i_2-i_1\leq cn}}\sum_{j_1\in [0, 2^{n})}\
   \iint\1_{B(y,
    r_{2^{n+1}})}(T^{i_1}_1x)\cdot \1_{B(T_2^{\tilde{j}}y,
    r_{2^{n+1}})}(T_1^{i_2}x)~d\mu_1(x)d\mu_2(y) \\
 &\quad \leq \sum_{\substack{i_1, i_2\in [0, 2^{n})\\ i_2-i_1\leq cn}}\sum_{j_1\in [0, 2^{n})}\
   \iint\1_{B(y,
    r_{2^{n+1}})}(T^{i_1}_1x)~d\mu_1(x)d\mu_2(y) \\
&\quad = \sum_{\substack{i_1, i_2\in [0, 2^{n})\\ i_2-i_1\leq cn}}\sum_{j_1\in [0, 2^{n})}\
   \int\mu_1(B(y,
    r_{2^{n+1}}))~d\mu_2(y) \leq cn2^{2n}2r_{2^{n+1}}.
  \end{align*}

So if $r_n \ge \frac{(\log n)^2(\log\log n)^{1+\eps}}{n^2}$, then
\begin{equation}\label{eq:RotSuff}
\frac{\E(\hat S_n^2)-\E(\hat S_n)^2}{\E(\hat S_n)^2} \leq \frac{2^{4n+1} C \mathrm{e}^{-n\theta c}+cn2^{2n}2r_{2^{n+1}}}{2^{4n+2}r^2_{2^{n+1}}}
\end{equation}
is summable since $c>4\log(2)/\theta$. 
\end{proof}

\begin{proof}[Proof of Theorem~\ref{thm:rotation}(\ref{thm:rotsupone})]
  Analogously to the proof of Theorem~\ref{thm:supinf2}
  (\ref{thm:supinf2-2}), it is enough to observe
  that~\eqref{eq:RotSuff} with $\hat{S}_n$ replaced by $S_n$ from
  \eqref{eq:S_n} goes to zero under the assumption
  $r_n=\log(n)h(n)/n^2$ with $h(n) \to \infty$. Furthermore, the
  Diophantine condition on $\alpha$ ensures
  $r_n = o( \Psi_{\alpha}(n))$ and, hence, the applicability of
  Claim~\ref{claim:Dioph} in this case.
\end{proof}

\section{Proofs for one orbit}

Suppose $X\subset \R^d$ is compact and $r>0$. Then we can cover $X$ by balls
$\{B(x_p,r)\}_{p = 1}^{k(r)}$ where $k(r) \asymp r^{-d}$ and there exists a $C_0\in\N$ such that each $x\in X$ belongs to at most $C_0$ elements of $\{B(x_p,2r)\}_{p = 1}^{k(r)}$.\footnote{For general metric spaces this property is known as \emph{bounded local complexity.}} 

We will write $\1_{p, r} := \1_{B(x_p, 2r)}$. For a proof of the following easy fact, see the ideas in \cite[Lemma~12]{GouRouSta24}.

\begin{lemma}
	Let $X\subset \R^d$ be compact. For all $x,y\in X$ we have,
	\begin{equation*}
			\1_{B(x,r)}(y) \le \sum_{p=1}^{k(r)}\1_{p,r}(x)\1_{p,r}(y)\le C_0\1_{B(x,4r)}(y).
	\end{equation*}
	Consequently, for a probability measure $\mu$, we have,
	\[
	\int\mu(B(x, r)) \, d\mu(x) \le \sum_{p=1}^{k(r)}\mu(B(x_p, 2r))^2\le C_0 \int\mu(B(x, 4r)) \, d\mu(x).
	\]
	\label{lem:GRS}
\end{lemma}
Recall the notation $A (r,n) = \{\, x : d (x,T^n x) < r \,\}$ and
the short return time estimate \eqref{eq:shortreturnestimate},
\begin{equation}
	\label{eq:shortreturnestimate2}
	\mu (A (r,n)) \leq C r^s.
\end{equation}
From this point on, let $X\subset \R$ denote an interval.  A
lemma by Kirsebom, Kunde and Persson
\cite[Lemma~3.1]{KirKunPer23}, implies that under exponential
mixing for BV against $L^1$ observables,
\begin{equation}
	\label{eq:KKPestimate}
	\mu (A(r,n)) \leq \int \mu (B(x,r)) \, \mathrm{d} \mu (x) + C_1
	e^{-\theta n}.
\end{equation}
for some constant $C_1>0$. 

\begin{proof}[Proof of Theorem~\ref{thm:1supinf}(\ref{thm:1orbinfzero})]
  Given a sequence $(r_n)_n$, let
  \[
    Q_n (x) := \sum_{0 \leq i < j < n} \1_{B(T^ix, r_n)} (T^j
    x).
  \]
  As argued in the proof of Theorem~\ref{the:measurezero}(\ref{infzero}), the result follows if $\E(Q_n)\to 0$.  We have
  that
  \[
    \E (Q_n) = \sum_{0 \leq i < j < n} \int \1_{B(T^ix, r_n)}
    (T^j x) \, d \mu (x)=\sum_{0 \leq i < j < n} \mu(A(r_n,j-i)).
  \]
  Using \eqref{eq:KKPestimate} when
  $j-i > 3 \theta^{-1} \log n$, and
  \eqref{eq:shortreturnestimate2} otherwise, we obtain
  \begin{align*}
    \sum_{0 \leq i < j < n} \mu(A(r_n,j-i))
    &=\sum_{\substack{0 \leq i < j < n\\j-i\leq 3\theta^{-1}\log
    n}} \mu(A(r_n,j-i)) + \!\!\! \sum_{\substack{0 \leq i < j < n\\j-i>
    3\theta^{-1}\log n}} \mu(A(r_n,j-i))\\
    &\leq C n( \log n) r_n^{s} + n^2 \int \mu (B(x,r_n)) \, d
      \mu(x) + \frac{C_1}{n}. 
  \end{align*}
  Hence $\E(Q_n) \to 0$ as $n \to \infty$. 
\end{proof}

\begin{proof}[Proof of Theorem~\ref{thm:1supinf}(\ref{thm:1orbsupzero})]
	Given a decreasing sequence $(r_n)_n$, let
	\begin{align*}
	\tilde{Q}_n(x)&:= \sum_{2^n \le j < 2^{n+1}} \sum_{0 \leq i < j} \mathbbm{1}_{B (T^i x, r_{2^n})} (T^j x).\\
	\tilde{Q} (x) &:= \sum_{n=0}^\infty \tilde{Q}_n(x).
	\end{align*}
	As argued in the proof of Theorem~\ref{the:measurezero}(\ref{supzero}), the result follows if $\mathbb{E} (\tilde{Q}) < \infty$.
	
	We have
	\begin{align*}
		\mathbb{E} (\tilde{Q})
		&= \sum_{n=0}^\infty \sum_{2^n \leq j <
			2^{n+1}} \sum_{0 \leq i < j} \int \mathbbm{1}_{B(T^i x,
			r_{2^n})} (T^j x) \, \mathrm{d} \mu \\
		&= \sum_{n=0}^\infty \sum_{2^n \leq j <
			2^{n+1}} \sum_{0 \leq i < j} \mu (A(r_{2^n},j-i)).
	\end{align*}
	As in part (\ref{thm:1orbinfzero}), we combine \eqref{eq:KKPestimate} with \eqref{eq:shortreturnestimate2} in
	order to estimate $\mathbb{E} (\tilde{Q})$. Take $c > 2/\theta$. We split
	the above sum as
	\begin{multline}
		\label{eq:ES-sums}
		\mathbb{E}(\tilde{Q}) \leq \sum_{n = 0}^\infty \sum_{2^n \leq j <
			2^{n+1}}
		\sum_{0 \leq i < j - c \log j} \mu (A(r_{2^n},j-i))\\
		+ \sum_{n = 0}^\infty \sum_{2^n \leq j < 2^{n+1}} \sum_{j - c
			\log j \leq i < j} \mu (A(r_{2^n},j-i)).
	\end{multline}
	For the first part we use \eqref{eq:KKPestimate} and obtain that
	\begin{multline*}
		\sum_{n = 0}^\infty \sum_{2^n \leq j < 2^{n+1}} \sum_{0 \leq
			i < j - c
			\log j} \mu (A(r_{2^n},j-i)) \\
		\leq \sum_{n=0}^\infty \sum_{2^n < j < 2^{n+1}} \sum_{0 \leq
			i < j - c \log j} \biggl( \int \mu (B (x,r_{2^n})) \,
		\mathrm{d} \mu (x) + C_1 e^{-\theta (j-i)} \biggr).
	\end{multline*}
	Then
	\[
	\sum_{n=0}^\infty \sum_{2^n \leq j < 2^{n+1}} \sum_{0 \leq i
		< j - c \log j} \int \mu (B (x,r_{2^n})) \, \mathrm{d} \mu
	(x) \leq \sum_{n=0}^\infty 2^{2n} \int \mu (B (x,r_{2^n})) \,
	\mathrm{d} \mu (x),
	\]
	which, since $n\int \mu(B(x,r_n))\, d\mu(x)$ is assumed decreasing, is finite by Cauchy condensation,
	Lemma~\ref{lem:CauchyC}. For the part containing the remainder
	term $C_1 e^{-\theta(j-i)}$ we estimate that
	\[
	\sum_{n=0}^\infty \sum_{2^n \leq j < 2^{n+1}} \sum_{0 \leq i
		< j - c \log j} C_1 e^{-\theta (j-i)} \leq \sum_{n=0}^\infty
	\sum_{2^n < j < 2^{n+1}} C_1 \frac{1}{j^2} \leq
	\sum_{n=0}^\infty C_1 2^{-n},
	\]
	which is also finite.
	
	We now turn to the second part of \eqref{eq:ES-sums}. Here we
	use \eqref{eq:shortreturnestimate2} to get
	\[
	\sum_{n = 0}^\infty \sum_{2^n \leq j < 2^{n+1}} \sum_{j - c
		\log j \leq i < j} \mu (A(r_{2^n},j-i)) \leq \sum_{n =
		0}^\infty \sum_{2^n \leq j < 2^{n+1}} c (\log j) C
	r_{2^n}^s.
	\]
	Using that $r_n \leq n^{-1/s} (\log n)^{-2/s - \varepsilon}$,
	we see that this is finite as well.
	
	These estimates together show that $\mathbb{E} (\tilde{Q}) < \infty$ and
	hence that $\mu (\limsup_n F_n) = 0$.
\end{proof}

\begin{proof}[Proof of Theorem~\ref{thm:1supinf}(\ref{thm:1orbinfone})]
	Let $\gamma\in (0,1/2)$ and define for a decreasing sequence $(r_n)_n$,
	\[
	\hat{Q}_n(x):=\sum_{\substack{i\in [0, \gamma2^{n})\\ j\in
			[(1-\gamma)2^{n}, 2^n)}}\sum_{p=1}^{k(r_{2^{n+1}})}
	\1_{p, r_{2^{n+1}}}(T^ix)\1_{p, r_{2^{n+1}}}(T^jx).
	\]
	For notational simplicity, set $L_n^\gamma:= [0,
        \gamma2^{n})$,  $R_n^\gamma:= [(1-\gamma)2^{n}, 2^n)$ and
        $\bar{r}:=r_{2^{n+1}}$. Suppose $\hat{Q}_n(x)\geq 1$ for
        some $n\in\N$. This means that there exists $(i,j)\in
        L_n^\gamma\times R_n^\gamma$ and an $x_p$ with $\1_{p,
          \bar{r}}(T^ix)\1_{p, \bar{r}}(T^jx)\geq 1$ which in
        turn implies that $d(T^i x,T^j x)\leq 4\bar{r}$. Since
        $(r_n)_n$ is decreasing, we then also have that
        $\hat{Q}_l(x)\geq 1$ for all $l\in[2^n,2^{n+1}]$. Hence
        if $\hat{Q}_n(x)\geq 1$ for all sufficiently large
        $n$ a.s., then $x\in\liminf_{n}F_{n,4r_n}^T$ a.s. As argued in the proof of Theorem~\ref{thm:supinf2} (\ref{thm:supinf2-1}), this will follow if $\sum_{n=1}^{\infty} \frac{\E(\hat{Q}_n^2)-\E(\hat{Q}_n)^2}{\E(\hat{Q}_n)^2}<\infty$.
	
	First we compute $\E(\hat{Q}_n)$.
	\begin{align*}
		\E(\hat{Q}_n)& = \sum_{i\in L_n^\gamma,  j\in
			R_n^\gamma}\sum_{p=1}^{k(\bar{r})} \int\1_{p,
			\bar{r}}(T^ix)\1_{p, \bar{r}}(T^jx) \,
		d\mu(x) \\  
		& = \sum_{i\in L_n^\gamma,  j\in
			R_n^\gamma}\sum_{p=1}^{k(\bar{r})}\left(\left( \int\1_{p,
			\bar{r}}(x) \, d\mu(x)\right)^2 + \mathrm{Err}(n)\right)\\ 
		&= \gamma^2 2^{2n}
		\bar{r}^{-d} \mathrm{Err}(n) +
		\gamma^22^{2n}\sum_{p=1}^{k(\bar{r})}\mu(B(x_p,
		2\bar{r}))^2.
	\end{align*}
	where $\mathrm{Err}(n)$ denotes the error which we can bound by $C'e^{-2^n(1-2\gamma)\theta'}$ using the exponential mixing estimate. Since $r_n\ge n^{-\beta}$ we see that the first term in the estimate above vanishes. For the second term we get from Lemma~\ref{lem:GRS} that
	\begin{equation}\label{eq:lemma51application}
		\int \mu (B(x,\bar{r})) \, d \mu(x) \leq
		\sum_{p=1}^{k(\bar{r})}\mu(B(x_p, 2\bar{r}))^2,
	\end{equation}
	so by the lower bound  assumption on $\int \mu (B(x,\bar{r})) \, d
	\mu(x)$, we have that
        \[
          2^{2n}\sum_{p=1}^{k(\bar{r})}\mu(B(x_p, 2\bar{r}))^2
          \to \infty \quad \text{as} \quad n \to \infty.
        \]
	
	Now we have 
	\begin{align*}
		\E(\hat{Q}_n^2)&= \sum_{\substack{i_1\in L_n^\gamma\\ j_1\in
				R_n^\gamma}} \sum_{\substack{i_2\in L_n^\gamma\\  j_2\in
				R_n^\gamma}}\sum_{p, q=1}^{k(\bar{r})} \int\1_{p,
			\bar{r}}(T^{i_1}x)\1_{p, \bar{r}}(T^{j_1}x)  \1_{q, \bar{r}}(T^{i_2}x)\1_{q,
			\bar{r}}(T^{j_2}x) \, d\mu(x).
	\end{align*} 
	
	We split the sum as in \cite[Section~4, Lower bound]{Zha24}.  We
	start with the totally separated case, were we suppose that
	$|i_2-i_1|> cn$ and $|j_2-j_1|> cn$, where we will choose $c>0$
	later: it is sufficient to assume that $i_2-i_1> cn$ and
	$j_2-j_1> cn$ since the other combinations lead to the same estimates.
	
	\begin{align*}
		& \sum_{\substack{i_1, i_2\in L_n^\gamma\\
				i_2-i_1>cn}}\sum_{\substack{j_1, j_2\in R_n^\gamma\\ j_2-j_1>
				cn}}\sum_{p,q} \int\1_{p, \bar{r}}(T^{i_1}x)\1_{p, \bar{r}}(T^{j_1}x)
		\1_{q, \bar{r}}(T^{i_2}x)\1_{q, \bar{r}}(T^{j_2}x) \, d\mu(x)\\ 
		& =  \sum_{\substack{i_1, i_2\in L_n^\gamma\\
				i_2-i_1>cn}}\sum_{\substack{j_1, j_2\in R_n^\gamma\\ j_2-j_1>
				cn}}\sum_{p,q} \int\1_{p, \bar{r}}(x) \1_{q, \bar{r}}(T^{i_2-i_1}x)\1_{p,
			\bar{r}}(T^{j_1-i_1}x) \1_{q, \bar{r}}(T^{j_2-i_1}x) \, d\mu(x)\\ 
		&\le \sum_{\substack{i_1, i_2\in L_n^\gamma\\
				i_2-i_1>cn}}\sum_{\substack{j_1, j_2\in R_n^\gamma\\ j_2-j_1>
				cn}}\sum_{p,q} \biggl( \int\1_{p, \bar{r}}(x)\1_{q,
			\bar{r}}(T^{i_2-i_1}x) \, d\mu(x)\int  \1_{p, \bar{r}}(x)\1_{q,
			\bar{r}}(T^{j_2-j_1}x) \, d\mu(x)\\ 
		&\hspace{9cm} + C' e^{-(j_1-i_2)\theta'} \biggr)\\
		&\le \sum_{\substack{i_1, i_2\in L_n^\gamma\\
				i_2-i_1>cn}}
		\sum_{\substack{j_1, j_2\in R_n^\gamma\\ j_2-j_1>cn}}
		\sum_{p,q}\left[ \left(\mu(B(x_p, 2\bar{r}))\mu(B(x_q, 2\bar{r}))+ C'e^{-n\theta' c}\right)^2+
		C' e^{-2^n(1-2\gamma)\theta'}\right]\\
		&\le  \gamma^42^{4n}\sum_{p,q}\left[ \left(\mu(B(x_p,
		2\bar{r}))\mu(B(x_q, 2\bar{r}))+ C'e^{-n\theta' c}\right)^2+  C' e^{-2^n(1-2\gamma)\theta'}\right].
	\end{align*}
	where the third line is by 4-mixing.  Since we assume that
	$r_n\ge n^{-\beta}$ and $k(r)\asymp r^{-d}$, up to constants this
	differs from $\E(\hat{Q}_n)^2$ by at most
	\begin{align*}
		\gamma^42^{4n}\sum_{p,q}(e^{-n\theta'
			c}+e^{-2^n(1-2\gamma)\theta'})
		&\le \gamma^42^{4n}k(\bar{r})^2(e^{-n\theta'
			c}+e^{-2^n(1-2\gamma)\theta'})\\ 
		&\le \gamma^42^{4n}2^{n2\beta d}(e^{-n\theta' c}+e^{-2^n(1-2\gamma)\theta'}).
	\end{align*}
	So if $c>2(2+\beta d)\log 2/\theta'$, then this is summable,
	without any conditions coming from dividing by $\E(\hat{Q}_n)^2$.
	
	Now consider the totally non-separated case, where it is
	sufficient to consider the case $0\le i_2-i_1\le cn$ and
	$0\le j_2-j_1\le cn$. Since any $x\in X$ is in at most $C_0$ of the covering sets $\{B(x_p,2r)\}_{p = 1}^{k(r)}$, we get $\sum_{q=1}^{k(\bar{r})} \1_{q, \bar{r}}(x)\cdot \1_{q, \bar{r}}(y)\le C_0$. Using this we have,
	\begin{align*}
		& \sum_{\substack{i_1, i_2\in L_n^\gamma\\0\le  i_2-i_1\le
				cn}}\sum_{\substack{j_1, j_2\in R_n^\gamma\\0\le j_2-j_1\le
				cn}}\sum_{p,q} \int\1_{p, \bar{r}}(T^{i_1}x)\1_{p, \bar{r}}(T^{j_1}x)
		\1_{q, \bar{r}}(T^{i_2}x)\1_{q, \bar{r}}(T^{j_2}x) \, d\mu(x)\\ 
		& \le C_0 \sum_{\substack{i_1, i_2\in L_n^\gamma\\0\le
				i_2-i_1\le cn}}\sum_{\substack{j_1, j_2\in R_n^\gamma\\0\le
				j_2-j_1\le  cn}}\sum_{p} \int\1_{p, \bar{r}}(T^{i_1}x)\1_{p,
			\bar{r}}(T^{j_1}x) \, d\mu(x)\\ 
		& \le C_0 \sum_{\substack{i_1, i_2\in L_n^\gamma\\0\le
				i_2-i_1\le cn}}\sum_{\substack{j_1, j_2\in R_n^\gamma\\0\le
				j_2-j_1\le  cn}}\sum_{p}( \mu(B(x_p, 2\bar{r}))^2+
		C'e^{-2^n(1-2\gamma)\theta'})\\ 
		&\le C_0\gamma^22^{2n}c^2n^2\sum_{p}( \mu(B(x_p,
		2\bar{r}))^2+ C'e^{-2^n(1-2\gamma)\theta'}) .
	\end{align*}
	So the important term to estimate here is
	$2^{2n}n^2\sum_p\mu(B(x_p, 2\bar{r}))^2\lesssim n^2\E(\hat{Q}_n),$
	which when divided by $\E(\hat{Q}_n)^2$ is $n^2 / \E(\hat{Q}_n)$. By using \eqref{eq:lemma51application} this is seen to be
	summable if we assume
	$\int\mu(B(x, r_n)) \, d\mu(x)\ge \frac{(\log n)^3(\log\log
		n)^{1+\eps}}{n^2}$. Note that this is a weaker requirement than the assumption made in the theorem. 
	
	Now for the half-separated case.  Let us suppose that
	$0\le i_2-i_1\le cn$ and $j_2-j_1 \geq cn$.  Then
	\begin{align*}
		& \sum_{\substack{i_1, i_2\in L_n^\gamma\\0\le  i_2-i_1\le
				cn}}\sum_{\substack{j_1, j_2\in R_n^\gamma\\ j_2-j_1>
				cn}}\sum_{p,q} \int\1_{p, \bar{r}}(T^{i_1}x)\1_{p, \bar{r}}(T^{j_1}x)
		\1_{q, \bar{r}}(T^{i_2}x)\1_{q, \bar{r}}(T^{j_2}x) \, d\mu(x)\\ 
		& = \sum_{\substack{i_1, i_2\in L_n^\gamma\\0\le  i_2-i_1\le
				cn}}\sum_{\substack{j_1, j_2\in R_n^\gamma\\ j_2-j_1>
				cn}}\sum_{p,q} \int\1_{p, \bar{r}}(x) \1_{q, \bar{r}}(T^{i_2-i_1}x)\1_{p,
			\bar{r}}(T^{j_1-i_1}x) \1_{q, \bar{r}}(T^{j_2-i_1}x) \, d\mu(x)\\ 
		&\le \sum_{\substack{i_1, i_2\in L_n^\gamma\\0\le  i_2-i_1\le
				cn}}\sum_{\substack{j_1, j_2\in R_n^\gamma\\ j_2-j_1>
				cn}}\sum_{p,q} \biggl( \int\1_{p, \bar{r}}(x) \1_{q,
			\bar{r}}(T^{i_2-i_1}x) \, d\mu(x)  \\
		&\hspace{5cm} \cdot \int\1_{p, \bar{r}}(x) \1_{q, \bar{r}}(T^{j_2-j_1}x)
		\, d\mu(x) + C' e^{-(j_1-i_2)\theta'} \biggr)\\
		&\le \sum_{\substack{i_1, i_2\in L_n^\gamma\\0\le  i_2-i_1\le
				cn}}\sum_{\substack{j_1, j_2\in R_n^\gamma\\ j_2-j_1>
				cn}}\sum_{p,q} \biggl (\int\1_{p, \bar{r}}(x) \1_{q, \bar{r}}(T^{i_2-i_1}x) \, d\mu(x) \\ 
		& \hspace{3cm} \cdot \left(\mu(B(x_p, 2\bar{r}))\mu(B(x_q, 2\bar{r}))+ C'e^{-n\theta' c}\right)+
		C' e^{-2^n(1-2\gamma)\theta'} \biggr), 
	\end{align*}
	where the third line is by 4-mixing.  Note that our choice of $c$
	again makes the part with the terms $C'e^{-n\theta' c}$ and
	$C' e^{-2^n (1-2\gamma)\theta'}$ summable over $n$.  Now by the
	Cauchy--Schwarz Inequality,
	\begin{align*}
		& \sum_{p,q} \int\1_{p, \bar{r}}(x) \1_{q,
			\bar{r}}(T^{i_2-i_1}x) \, d\mu(x)\mu(B(x_p,
		2\bar{r}))\mu(B(x_q, 2\bar{r}))\\ 
		&= \int\sum_p\mu(B(x_p, 2\bar{r}))\1_{p, \bar{r}}(x)\sum_q
		\mu(B(x_q, 2\bar{r}))\1_{q, \bar{r}}(T^{i_2-i_1}x) \, d\mu(x)\\  
		&\le \biggl( \int \Bigl(\sum_p\mu(B(x_p, 2\bar{r}))\1_{p,
			\bar{r}}(x) \Bigr)^2 \, d\mu(x) \biggr)^{\frac12}  \\
		& \hspace{5cm} \cdot \biggl( \int \Bigl( \sum_q \mu(B(x_q,
		2\bar{r})) \1_{q, \bar{r}}(T^{i_2-i_1}x) \Bigr)^2 \, d\mu(x)
		\biggr)^{\frac12}\\ 
		&=\int \Bigl(\sum_p\mu(B(x_p, 2\bar{r}))\1_{p,
			\bar{r}}(x) \Bigr)^2 \, d\mu(x).
	\end{align*}
	Recall that for each $x\in X$, the sum
        $\sum_p\mu(B(x_p, 2\bar{r}))\1_{p, \bar{r}}(x)$ contains
        at most $C_0$ non-zero terms. For these non-zero terms we
        apply the following consequence of Jensen's inequality
	\[
          (a_1+ \cdots + a_m)^2 \le m(a_1^2 + \cdots + a_m^2),
	\]
	for numbers $a_1, \ldots, a_m\ge 0$ along with the fact
        that $\1_{p, \bar{r}}^2= \1_{p, \bar{r}}$ to get the
        bound
        $C_0\sum_p\mu(B(x_p, 2\bar{r}))^2\1_{p, \bar{r}}(x)$.
        Hence, by superadditivity of convex functions, the above
        sum is bounded by,
	\begin{align*}
	\int C_0\sum_p\mu(B(x_p, 2\bar{r}))^2\1_{p,
          \bar{r}}(x) \, d\mu(x)
          & = C_0\sum_p\mu(B(x_p, 2\bar{r}))^3 \\
          & \le	C_0 \left( \sum_p\mu(B(x_p, 2\bar{r}))^2
            \right)^{\frac32}.
	\end{align*}
	So the entire sum we must estimate is
	\[
	2^{3n}cn \left( \sum_p\mu(B(x_p, 2\bar{r}))^2
	\right)^{\frac32}\lesssim n \E(\hat{Q}_n)^{\frac32},
	\]
	which when divided by $\E(\hat{Q}_n)^2$ can be estimated
        as $n / \E(\hat{Q}_n)^{\frac12}$.  Hence choosing
        $\int\mu(B(x, r_n)) \, d\mu(x)\gtrsim \frac{(\log
          n)^4(\log\log n)^{2+\eps}}{n^2}$, this is bounded, up to constants, by
	\[
	\frac{n}{(n^4(\log n+\log \log 2)^{2+\eps}(\log
		2)^4)^{\frac12}},
	\]
	which is summable.
\end{proof}

\begin{proof}[Proof of Theorem~\ref{thm:1supinf}(\ref{thm:1orbsupone})]  The proof follows the same principle
  as the proof of Theorem~\ref{thm:supinf2}(\ref{thm:supinf2-2}). We define
  \[
    \bar{Q}_n(x):=\sum_{\substack{i\in [0, \gamma n)\\ j\in
        [(1-\gamma) n, n)}}\sum_{p=1}^{k(r_n)} \1_{p,
      r_n}(T^ix)\1_{p, r_n}(T^jx).
  \]
  and observe that it is sufficient to prove that
  $\frac{\E(\bar{Q}_n^2)-\E(\bar{Q}_n)^2}{\E(\bar{Q}_n)^2}$ goes
  to 0. Again, all estimates and computations of part
  (\ref{thm:1orbinfone}) are repeated with $2^n$ replaced by $n$
  and the separation gap $cn$ replaced by $c\log n$. Using the
  corresponding assumptions on $r_n$ imposed in part
  (\ref{thm:1orbsupone}) of the theorem one easily reaches the
  conclusion that the aforementioned quantity vanishes.
\end{proof}

\begin{remark}
  In part (\ref{thm:1orbinfone}) and (\ref{thm:1orbsupone}) the
  assumptions on $(r_n)_n$ have some flexibility in the following
  sense. The condition that $r_n\ge n^{-\beta}$ can be relaxed to
  $(r_n)_n$ decreasing at most subexponentially at the cost of
  increasing the power of $\log n$ by an arbitrarily small amount
  in the lower bound on the shrinking rate of
  $\int\mu(B(x, r_n)) \, d\mu(x)$. In the proof, this would be
  reflected by replacing our time gap $cn$ with $n^{1+\iota}$ for
  $\iota>0$ in part (\ref{thm:1orbinfone}) and $c\log n$ by
  $(\log n)^{1+\iota}$ for $\iota>0$.
\end{remark}

\end{document}